\documentclass[final,onefignum,onetabnum]{siamart190516}

\usepackage{braket,amsfonts}

\usepackage{array}
\usepackage{subcaption}
\usepackage{caption}
\usepackage{enumitem}

\usepackage{pgfplots}

\newsiamthm{claim}{Claim}
\newsiamremark{remark}{Remark}
\newsiamremark{hypothesis}{Hypothesis}
\crefname{hypothesis}{Hypothesis}{Hypotheses}

\usepackage{algorithmic}

\usepackage{graphicx,epstopdf}

\Crefname{ALC@unique}{Line}{Lines}

\usepackage{amsopn}

\usepackage{xspace}
\usepackage{bold-extra}
\usepackage[most]{tcolorbox}

\colorlet{texcscolor}{blue!50!black}
\colorlet{texemcolor}{red!70!black}
\colorlet{texpreamble}{red!70!black}
\colorlet{codebackground}{black!25!white!25}

\newcommand{\argmin}{\text{arg}\min}
\newcommand{\argmax}{\text{arg}\max}

\usepackage{hyperref}
\usepackage{epsfig}
\usepackage[norelsize,boxed,linesnumbered,vlined,algo2e]{algorithm2e} 
\graphicspath{{figures/}}
\usepackage{tikz}
\usetikzlibrary{shapes,arrows}

\usepackage{algorithmic}

\usepackage{media9}
\usepackage{graphicx}

\lstdefinestyle{siamlatex}{%
  style=tcblatex,
  texcsstyle=*\color{texcscolor},
  texcsstyle=[2]\color{texemcolor},
  keywordstyle=[2]\color{texemcolor},
  moretexcs={cref,Cref,maketitle,mathcal,text,headers,email,url},}

\tcbset{%
  colframe=black!75!white!75,
  coltitle=white,
  colback=codebackground, 
  colbacklower=white, 
  fonttitle=\bfseries,
  arc=0pt,outer arc=0pt,
  top=1pt,bottom=1pt,left=1mm,right=1mm,middle=1mm,boxsep=1mm,
  leftrule=0.3mm,rightrule=0.3mm,toprule=0.3mm,bottomrule=0.3mm,
  listing options={style=siamlatex}
}

\newtcblisting[use counter=example]{example}[2][]{%
  title={Example~\thetcbcounter: #2},#1}

\newtcbinputlisting[use counter=example]{\examplefile}[3][]{%
  title={Example~\thetcbcounter: #2},listing file={#3},#1}

\DeclareTotalTCBox{\code}{ v O{} }
{ 
  fontupper=\ttfamily\color{black},
  nobeforeafter,
  tcbox raise base,
  colback=codebackground,colframe=white,
  top=0pt,bottom=0pt,left=0mm,right=0mm,
  leftrule=0pt,rightrule=0pt,toprule=0mm,bottomrule=0mm,
  boxsep=0.5mm,
  #2}{#1}

\patchcmd\newpage{\vfil}{}{}{}
\begin{tcbverbatimwrite}{tmp_\jobname_header.tex}
\title{Mean field control problems for Vaccine distribution\thanks{Submitted to the editors \today
\funding{This work is supported by AFOSR MURI FA9550-18-1-0502.}}}

\author{{Wonjun Lee}
\thanks{Department of Mathematics, University of California, Los Angeles (\email{wlee@math.ucla.edu}, \email{siting6@math.ucla.edu},  \email{sjo@math.ucla.edu.})}
\and {Siting Liu}\footnotemark[2]
\and {Wuchen Li}
\thanks{Department of Mathematics, University of South Carolina (\email{wuchen@mailbox.sc.edu})}
\and {Stanley Osher}\footnotemark[2]}

\headers{Mean field control problems for vaccine distribution}{Lee, Liu, Li and Osher}
\end{tcbverbatimwrite}
\title{Mean field control problems for Vaccine distribution\thanks{Submitted to the editors \today
\funding{This work is supported by AFOSR MURI FA9550-18-1-0502.}}}

\author{{Wonjun Lee}
\thanks{Department of Mathematics, University of California, Los Angeles (\email{wlee@math.ucla.edu}, \email{siting6@math.ucla.edu},  \email{sjo@math.ucla.edu.})}
\and {Siting Liu}\footnotemark[2]
\and {Wuchen Li}
\thanks{Department of Mathematics, University of South Carolina (\email{wuchen@mailbox.sc.edu})}
\and {Stanley Osher}\footnotemark[2]}

\headers{Mean field control problems for vaccine distribution}{Lee, Liu, Li and Osher}

\ifpdf
\hypersetup{ pdftitle={Mean field control problems for vaccine distribution} }
\fi

\allowdisplaybreaks

\begin{document}

\maketitle
\begin{abstract}
With the invention of the COVID-19 vaccine, shipping and distributing are crucial in controlling the pandemic. In this paper, we build a mean-field variational problem in a spatial domain, which controls the propagation of pandemics by the optimal transportation strategy of vaccine distribution. Here, we integrate the vaccine distribution into the mean-field SIR model designed in \cite{lee2020controlling}. Numerical examples demonstrate that the proposed model provides practical strategies for vaccine distribution in a spatial domain.
\end{abstract}
\section{Introduction}
The COVID-19 pandemic has affected society significantly.
Various actions are taken to mitigate the spread of the infections, such as the travel ban, social distancing, and mask-wearing.
The recent invention of the vaccine yields breakthroughs in fighting against this infectious disease. 
According to the recent effectiveness study~\cite{Francis2021efficacy}, vaccines including Pfizer, Moderna, and Janssen (J\&J) show approximately $66\%$-$95\%$ efficacy at preventing both mild and severe symptoms of COVID-19.
Therefore, the deployment of COVID-19 vaccines is an urgent and timely task.
Many countries have implemented phased distribution plans that prioritize the elderly and healthcare workers getting vaccinated. 
Meanwhile, the shipping of vaccines is expensive due to the cold chain transportation~\cite{Lin2020Vaccine}.
An effective distribution strategy is necessary to eliminate infectious diseases and prevent more death.

In this work, we propose a novel mean-field control model based on \cite{lee2020controlling}. 
We consider two approaches (controls) to control the pandemic:  relocation of populations and distribution of vaccines. The first one has been discussed thoroughly in~\cite{lee2020controlling}, where we address the spatial effect in pandemic modeling by introducing a mean-field control problem into the spatial SIR model. By applying spatial velocity to the classical disease model, the model finds the most optimal strategy to relocate the different populations (susceptible, infected, and recovered), controlling the epidemic's propagation.
We considered several aspects of the vaccine in our model for vaccine distribution, including manufacturing, delivery, and consumption.
Our goal is to find an optimal strategy to move the population and distribute vaccines to minimize the total number of infectious, the amount of movement of the people, and the transportation cost of the vaccine with limited vaccine supply. 
To tackle this question, we ensemble these two controls and propose the following constrained optimization problem: 

\begin{equation*}
        \min_{(\rho_i,v_i)_{i\in\{S,I,R,V\}},f} \;{G}\left((\rho_i,v_i)_{i\in\{S,I,R,V\}},f\right) \quad \text{($G$ defined from \eqref{eq:var_vacc})}\\
\end{equation*}
subject to 
\begin{equation*}
  \left\{ \begin{aligned}
       & \partial_t \rho_S + \nabla \cdot (\rho_S v_S) = - \beta \rho_S K * \rho_I + \frac{\eta_S^2}{2} \Delta \rho_S - \theta_1 \rho_V \rho_S && (t,x)\in(0,T)\times\Omega\\
        &\partial_t \rho_I + \nabla \cdot (\rho_I v_I) = \beta \rho_S K * \rho_I - \gamma \rho_I + \frac{\eta^2_I}{2} \Delta \rho_I && (t,x) \in(0,T)\times\Omega\\
        &\partial_t \rho_R + \nabla \cdot (\rho_R v_R) = \gamma \rho_I + \frac{\eta^2_R}{2} \Delta \rho_R + \theta_1 \rho_V \rho_S && (t,x) \in(0,T)\times\Omega\\
        &\partial_t \rho_V = f(t,x) - \theta_2 \rho_V \rho_S && (t,x) \in(0,T')\times\Omega\\
        &\partial_t \rho_V + \nabla \cdot (\rho_V v_V)  = - \theta_2 \rho_V \rho_S && (t,x) \in[T',T)\times\Omega\\
    \end{aligned}\right.
\end{equation*}
and
 \begin{equation*}
  \left\{    \begin{aligned}
        &0 \leq f(t,x) \leq f_{max} && (t,x) \in [0,T'] \times \Omega_{factory}\\    
        &f(t,x) = 0 && (t,x) \in [0,T'] \times \Omega \backslash \Omega_{factory}\\
        & \rho_V(t,x) \leq C_{factory} && (t,x) \in [0,T'] \times \Omega_{factory}
    \end{aligned}\right.
 \end{equation*}

In our model, different populations are described using $\rho_i$ ($i\in\{S, I, R\}$), representing the susceptible, infectious, and recovered. The term $\rho_V(x,t)$ describes the density distribution of the vaccine over the spatial domain at location $x$ and time $t$. 
The control variables $v_i$ ($i\in\{S, I ,R \}$) create velocity fields over time-space domain that move the corresponding populations.
As for vaccines, the control variable $v_V$ represents the vaccine's transportation strategy, and the control variable $f(t,x)$ describes how many vaccines are produced at a specific time and location.  
The optimization objective function $G$ is the sum of terminal costs $\mathcal{E}_{final}$ and running costs  $\mathcal{E}_{running}$. 
The terminal costs $\mathcal{E}_{final}$ represent the goal of our control to achieve at the terminal time, such as minimizing the total number of infectious individuals and maximizing the total number of recovered (immune) persons. The running costs $\mathcal{E}_{running}$ include the costs of transportation of vaccines and different classes of the populations, etc. We will discuss more details of cost functionals in Section~\ref{subsection:cost-function}.
As for constraints of our optimization problem, the five partial differential equations of $\rho_i$, $v_i$ ($i\in\{S, I, R, V\}$) describe the dynamics of the different classes of population and vaccines in terms of densities and velocities. The inequalities of $f(t,x)$ model the limitation of vaccine manufacturing. Vaccines are produced at particular factory locations $\Omega_{factory}$ with a daily maximal production rate $f_{max}$.
The dynamics of the vaccine density $\rho_V$ share some similar aspects to the unnormalized optimal transport~\cite{lee2021generalized}. Specifically, they both study mass transportation with a source term that creates masses.

{
We solve the main problem using the algorithm based on the first-order Primal-Dual Hybrid Gradient (PDHG) method~\cite{champock11,champock16}. Due to the multiplicative interaction terms, $\rho_S K*\rho_I$, $\rho_I K*\rho_S, \rho_V\rho_S$, the optimization problem is based on nonlinear PDE constraints, whereas the PDHG only considers linear constraints. We use the extension of the PDHG~\cite{clason2017primal} that solves nonsmooth optimization problems with nonlinear operators between function spaces. We extend the method utilizing the preconditioning operator from~\cite{JacobsLegerLiOsher2018_solvinga} which provides a suitable choice of variable norms to achieve a convergence rate independent of the nonlinear operator. As a result, the algorithm converges to the saddle point locally with step length parameters independent of the finite-difference mesh size; see Section~\ref{section:algorithms-subsection:properties} for details.
}

Lots of mathematical models have been invented to predict the future of COVID-19 epidemics. Recently proposed models take more real-world situations into consideration and tend to be more effective in quantitative forecasting. Specifically, there have been studies on the impact of actions such as lockdown, social distancing, wearing a mask \cite{dimarco2020social,di2020impact,flaxman2020estimating}.
Data-driven approach and machine learning techniques are also integrated to estimate the parameters for the epidemic better and boost the prediction of the trend of the pandemic model \cite{sesterhenn2020adjoint, ndiaye2020analysis}.
Meanwhile, optimal control serves as an important tool in pandemic control. 
They seek the optimal strategy to minimize the total number of infected people while keeping certain costs at a minimum.
There are work focused on mitigating the epidemic with limited medical supply, such as ICU capacity \cite{charpentier2020covid}, face masks \cite{liu2020optimal}, and vaccines \cite{zaman2008stability,hansen2011optimal,kim2016constrained,libotte2020determination,jang2020optimal}. 
In \cite{jang2020optimal}, an optimal vaccine distribution strategy is proposed with a limited total amount of vaccines and maximal daily supply.
\cite{libotte2020determination} first uses an inverse problem to determine the parameters of the SIR model. Then it formulates two optimal control problems, with mono- and multi-objective, and solves for the optimal strategy of vaccine administration. 
Other non-pharmaceutical interventions are also considered in the scope of optimal control of epidemics, including social distancing, closing schools, and lockdown \cite{godara2021control,kantner2020beyond,silva2021optimal}.
\cite{kantner2020beyond} computes the optimal non-pharmaceutical intervention strategy based on an extended SEIR model with the absence of the vaccine.
The mean-field control problem can be viewed as a particular type of optimal control applied to an individual in terms of population density. 

Mean-field game (control), introduced by \cite{huang2006large,lasry2007mean}, describes the deterministic (stochastic) differential games as the number of players tends to infinity, where a given player interacts through the distribution of all players in the state-space. 
It is a thriving research direction with applications in economics, crowd motion, industrial engineering, and more~\cite{aurell2020optimal,doncel2020mean,laguzet2015individual}.
Numerical methods are invented to obtain quantitative information of such mean-field game (control) models, especially when the state-space is in high dimensions \cite{achdou2010mean,achdou2020mean,briceno2019implementation,ruthotto2020machine}.
Multi-population mean-field game (control) problems have also drawn lots of attention \cite{bensoussan2018mean,cirant2015multi,feleqi2013derivation}. This type of problem studies the interactions on two levels: between agents of the same population and between populations.
Our model is a multi-population mean-field control problem with population dynamics described using reaction-diffusion equations adopted from the epidemic model and the controls over the vaccine production and distribution.
Therefore, we obtain a novel mean-field control problem.

The rest of the paper is organized as follows. Section~\ref{section:models} proposes a novel multi-population mean-field control model and explains how population movement and vaccine distribution are integrated into a constrained optimization problem.
Section~\ref{section:algorithms} discusses the challenges in numerically solving this mean-field control model, proposes a first-order primal-dual algorithm to solve it, and shows the local convergence of the algorithm.
Lastly, in Section~\ref{sec:experiments}, we present numerical experiments with different model parameter choices and discuss their implications on mean-field controls. 

\section{Models}\label{section:models}
{In this section, we review the classical SIR model. Based on it, we formulate the spatial SIR dynamics with vaccine distribution, namely SIRV dynamics. We then introduce a variational problem to control the SIRV dynamics.}
\subsection{Classical SIR model}
The SIR epidemic model describes an infectious disease epidemic via an ordinary differential equation system 
\begin{equation*}
\left\{    \begin{aligned}
    &\frac{dS}{dt} = -\beta IS\\
    &\frac{dI}{dt} = \beta IS - \gamma I \\
    &\frac{dR}{dt} = \gamma I .
    \end{aligned}\right.
\end{equation*}
The population is divided into three classes: susceptible, infected and recovered. While assuming a closed population without births or deaths, the model uses $S(t), I(t)$, and $R(t)$ to represent the number in each compartment at time $t$.
The SIR model has two parameters:
$\beta$ is the effective contact rate of the susceptible individual being infected and
$\gamma$ is the recovery rate of the infected individual.
The simplicity of this model allows people to predict an infectious disease epidemic by only estimating a few parameters. 
However, it has limitations by assuming the population is homogeneous-mixing, which means that every individual has an equal probability of disease-causing contact.
As a result, the predictions will lack spatial information and may not help the (local) governments make policies or relocate medical resources.
Therefore, we are motivated to study the spatial SIR model.
On the other hand, the SIR model does not consider the latent period between when a person is exposed to a disease and when they become infected.
This leads to the extension of the SIR model, such as the SEIR model.
Our proposed model has a flexible structure and can naturally be generalized to such epidemiological models.

\subsection{Spatial SIR variational problem with vaccine distribution}

In \cite{lee2020controlling}, we add the spatial dimension to the $S$, $I$, $R$ functions. Let $\Omega\subset \mathbb{R}^d$ be a bounded domain. Consider the following density functions
\begin{equation*}
    \rho_S, \rho_I, \rho_R:\ [0,T]\times \Omega \rightarrow [0,\infty).
\end{equation*}
Here, $\rho_S$, $\rho_I$, and $\rho_R$ represent susceptible, infected, and recovered populations distribution, respectively. 
We assume $\rho_i$ for each $i\in \{S,I,R\}$ moves over a spatial domain $\Omega$ with a velocity $v_i$. Here $v_i, i\in \{S,I,R\}$ are our controls variables. With change of variables $m_i=\rho_i v_i$, we define the momentum
\begin{equation*}
m_S, m_I, m_R \colon [0,T]\times \Omega \rightarrow\mathbb{R}^d
\end{equation*}
that govern the corresponding density flows.
In the following, instead of using control variables $v_i$, we replace them with $\frac{m_i}{\rho_i}$ and regard $m_i$ as the control variables.

We can describe the flows of the densities by the following continuity equations.
\begin{equation}\label{eq:continuity_orig}
  \left\{ \begin{aligned}
       & \partial_t \rho_S + \nabla \cdot m_S = - \beta \rho_S K * \rho_I + \frac{\eta_S^2}{2} \Delta \rho_S\\
        &\partial_t \rho_I + \nabla \cdot m_I = \beta \rho_I K * \rho_S - \gamma \rho_I + \frac{\eta^2_I}{2} \Delta \rho_I\\
        &\partial_t \rho_R + \nabla \cdot m_R = {\gamma \rho_I} + \frac{\eta^2_R}{2} \Delta \rho_R\\
        & \rho_S(0,\cdot), \rho_I(0,\cdot), \rho_R(0,\cdot) \text{ are given.}
    \end{aligned}\right.
\end{equation}
This system of continuity equations describes the flows of three groups of densities while satisfying the SIR model. The nonnegative constants $\eta_i$ ($i\in \{S,I,R\}$) are the coefficients for viscosity terms. These terms can also be understood as noise terms generated by the data. $K=K(x,y)$ is a symmetric positive definite kernel with $(K*\rho)(x,t) = \int_{\Omega}K(x,y)\rho(y,t)\,dy$. In this model, we consider the Gaussian kernel
\begin{equation*}
    \begin{aligned}
        K(x,y) &= \frac{1}{\sqrt{(2\pi)^d}} \prod^d_{k=1} \frac{1}{\sigma_k} \exp{\left(-\frac{|x_k-y_k|^2}{2\sigma_k^2}\right)}.
    \end{aligned}
\end{equation*}
The kernel convolution describes the spreading rate of infectious disease over the spatial domain. In addition, we assume the Neumann boundary conditions on $\partial \Omega$. Since we don't consider birth or death in our model, the total population is conserved for all time $t\in[0, T]$, which leads to the following equality
\[
    \frac{\partial}{\partial t} \int_\Omega \rho_S(t,x)+\rho_I(t,x)+\rho_R(t,x) dx = 0.
\]

In this paper, we consider the optimization problem for the distribution of vaccines. We add an extra function $\rho_V:[0,T]\times\Omega\rightarrow[0,\infty)$ which represents the vaccine density in $\Omega$ at each time $t\in[0,T]$. The vaccine distribution will be described as the following PDE:
\begin{equation}\label{eq:vaccine-pde}
    \begin{aligned}
        &\partial_t \rho_V = f(t,x) - \theta_2 \rho_V \rho_S && t\in(0,T')\\
        &\partial_t \rho_V + \nabla \cdot m_V = - \theta_2 \rho_V \rho_S && t\in[T',T), \quad 0<T'<T,
    \end{aligned}
\end{equation}
where $m_V:[T',T)\times\Omega\rightarrow\mathbb{R}^d$ is a momentum, $\theta_2$ represents the utilization rate of vaccines, and $f:(0,T')\times\Omega\rightarrow[0,\infty)$ represents the production rate of vaccines in $x \in \Omega$ at $0<t<T'$. During $0<t<T'$, the vaccines are produced with a production rate $f$ and used at a rate $\theta_2 \rho_V\rho_S$. During $T'\leq t < T$, the vaccines are delivered to the area where the susceptible population is located, and they are used at a rate of $\theta_2 \rho_V \rho_S$. In summary, the first part of the PDE describes vaccines' production, and the second part describes the delivery of vaccines. For all time $0<t<T$, the susceptible population is vaccinated if the vaccines are available in the same area. Now we are ready to introduce the new system of equations for the SIRV model.
\begin{equation}\label{eq:continuity_vacc}
  \left\{ \begin{aligned}
       & \partial_t \rho_S + \nabla \cdot m_S = - \beta \rho_S K * \rho_I + \frac{\eta_S^2}{2} \Delta \rho_S - \theta_1 \rho_V \rho_S && (t,x)\in(0,T)\times\Omega\\
        &\partial_t \rho_I + \nabla \cdot m_I = \beta \rho_S K * \rho_I - \gamma \rho_I + \frac{\eta^2_I}{2} \Delta \rho_I && (t,x) \in(0,T)\times\Omega\\
        &\partial_t \rho_R + \nabla \cdot m_R = \gamma \rho_I + \frac{\eta^2_R}{2} \Delta \rho_R + \theta_1 \rho_V \rho_S && (t,x) \in(0,T)\times\Omega\\
        &\partial_t \rho_V = f(t,x) - \theta_2 \rho_V \rho_S && (t,x) \in(0,T')\times\Omega\\
        &\partial_t \rho_V + \nabla \cdot m_V  = - \theta_2 \rho_V \rho_S && (t,x) \in[T',T)\times\Omega\\
        & \rho_S(0,\cdot), \rho_I(0,\cdot), \rho_R(0,\cdot), \rho_V(0,\cdot) \text{ are given.}
    \end{aligned}\right.
\end{equation}
In the first and third equations, we add the terms $ - \theta_1 \rho_V \rho_S$ and $ + \theta_1 \rho_V \rho_S$, respectively. The constant $\theta_1$ represents the vaccine efficiency and $\theta_1 \rho_V(t,x) \rho_S(t,x)$ represents the vaccinated population at $(t,x) \in (0,T)\times\Omega$. We denote a set $\mathbb{S} := \{S,I,R,V\}$ and define a nonlinear operator $A$ as follows
\begin{equation}\label{eq:nonlinear-operator-A}
\begin{split}
    A((\rho_i,m_i)_{i\in\mathbb{S}}, f) := (& \partial_t \rho_S + \nabla \cdot m_S - \frac{\eta_S^2}{2} \Delta \rho_S + \beta\rho_S K * \rho_I + \theta_1 \rho_S \rho_V,\\
              & \partial_t \rho_I + \nabla \cdot m_I - \frac{\eta_I^2}{2} \Delta \rho_I - \beta\rho_S K * \rho_I + \gamma \rho_I,\\
              &\partial_t \rho_R + \nabla \cdot m_R - \frac{\eta_R^2}{2} \Delta \rho_R  - \gamma \rho_I - \theta_1 \rho_S \rho_V,\\
              &\partial_t \rho_V - f \mathcal{X}_{[0,T')}(t) + \nabla \cdot m_V \mathcal{X}_{[T',T]}(t) + \theta_2 \rho_S \rho_V ),
\end{split}
\end{equation}
where $\mathcal{X}_C:[0,T]\rightarrow \mathbb{R}$ is a step function that equals $1$ on $C$ and $0$ otherwise. 

\subsection{The cost functional}\label{subsection:cost-function}

The cost functional we propose in this paper is the extension of~\cite{lee2020controlling}. We design the cost functional so that the solution $(\rho_i, m_i)$, $i\in\mathbb{S}$ satisfies the following criteria:

\medskip

\begin{enumerate}[label=(\roman*)]
    \item minimize the transportation cost for moving each population; \label{item1}
    \item minimize the total number of infected people and the total number of susceptible people by maximizing the usage of the vaccines at time $T$; \label{item2}
    \item maximize the total number of recovered people at time $T$; \label{item6}
    \item avoid high concentration of population and vaccines at each time $t\in(0,T)$; \label{item3}
    \item minimize the amount of vaccines produced during $t \in (0,T')$; \label{item4}
    \item minimize the transportation cost for delivering vaccines during $t \in (T',T)$. \label{item5}
\end{enumerate}

\medskip

Item \ref{item1} can be described by
\[
    \int^T_0 \int_\Omega F_i(\rho_i(t,x), m_i(t,x)) dx\, dt,
\]
for $i\in\{S,I,R\}$ where
\begin{equation}\label{eq:kinetic}
    F_i (\rho_i, m_i) =
    \begin{cases}
        \frac{\alpha_i |m_i|^2}{2 \rho_i} & \text{if } \rho_i > 0\\
        0 & \text{if } \rho_i = 0 \text{ and } |m_i| = 0\\
        \infty & \text{if } \rho_i = 0 \text{ and } |m_i| > 0,
    \end{cases}
\end{equation}
which is convex, lower semi-continuous, and $1$-homogeneous with respect to $(\rho_i,m_i)$.
The parameter $\alpha_i$ characterizes the cost of moving $\rho_i$ with velocity $\frac{m_i}{\rho_i}$. Larger $\alpha_i$ means it is more expensive to move $\rho_i$. Note that this function comes from the quadratic kinetic energy. To see this, we use the definition $m_i = \rho_i v_i$ and plug into the formula~\eqref{eq:kinetic}:
\[
    F_i(\rho_i,m_i) = \frac{\alpha_i |m_i|^2}{2 \rho_i} = \frac{\alpha_i}{2} \rho_i |v_i|^2.
\]

Item \ref{item2} and \ref{item6} can be described by the terminal costs of the cost functional
\begin{align*}
    \mathcal{E}_i(\rho_i(T,\cdot)) &= \int_\Omega e_i(\rho_i(T,x))\, dx \quad (i=S,I,V),\\
    \mathcal{E}_R(\rho_R(T,\cdot)) &= \int_\Omega e_R\left(1 - \rho_R(T,x)\right)\, dx,
\end{align*}
where functions $e:[0, \infty)\rightarrow [0, \infty)$ are convex and lower semi-continuous functions. We also minimize the terminal cost for $\rho_V$ because maximizing the usage of vaccines is equivalent to minimizing the number of vaccines left at the terminal time $T$. The total number of the recovered can be maximized by penalizing the density at the terminal time if the value of $\rho_R(T,x)$ is far away from $1$ for $x\in\Omega$. In this paper, we use a quadratic cost function
\begin{equation}\label{eq:item:e-quadratic}
    e_i(t) = \frac{a_i}{2}t^2, \quad (t\in[0,\infty))
\end{equation}
where $a_i$ is some constant.

For Item \ref{item3}, the cost functional for the concentration of the total population and vaccines can be represented by
\[
    \int^T_0 \mathcal{G}_P(\rho_S(t,\cdot) + \rho_I(t,\cdot) + \rho_R(t,\cdot))\,dt,\quad \int^T_0 \mathcal{G}_V(\rho_V(t,\cdot))\,dt,
\]
where
\begin{equation}\label{eq:definition-I}
    \mathcal{G}_P(u) = \int_\Omega g_P(u(x))\,dx,\quad \mathcal{G}_V(u) = \int_\Omega g_V(u(x))\,dx,
\end{equation}
for $u:\Omega\rightarrow [0,\infty)$ and convex and lower semi-continuous functions $g_P, g_V:[0,\infty)\rightarrow[0,\infty)$. Similar to $e_i$~\eqref{eq:item:e-quadratic} from Item~\ref{item2}, we use quadratic functions for $g_P$ and $g_V$.

Items~\ref{item4} and~\ref{item5} are criteria specific to the vaccine distribution. From the PDE~\eqref{eq:vaccine-pde}, the vaccines are produced during $0<t<T'$ by a function $f$. We use the similar functional~\eqref{eq:definition-I} to minimize the amount of vaccines produced by $f$. Thus, we set the functional
\[
    \int^{T'}_0 \mathcal{G}_0 (f(t,\cdot)) \, dt = \int^{T'}_0 \int_\Omega g_0(f(t,x))\, dx \, dt
\]
where $g_0:[0,\infty)\rightarrow[0,\infty)$ is a convex and lower semi-continuous function.

The vaccines are delivered during $T' < t< T$. Similar to the Item~\ref{item1}, we set
\[
    \int^T_{T'} \int_\Omega F_V(\rho_V, m_V) \, dx\,dt,
\]
where $F_V$ has the same definition as \eqref{eq:kinetic}.

The total cost functional we consider is then
\begin{equation}\label{eq:var_vacc}
    \begin{split}
        {G}((\rho_i,m_i)_{i\in\mathbb{S}},f) &= \sum_{i\in\mathbb{S}} \mathcal{E}_i(\rho_i(T,\cdot)) \\
        & + \int^T_0  \int_\Omega \sum_{i = S,I,R} F_i(\rho_i, m_i) \,dx\, dt +\int^T_{T'} \int_\Omega F_V(\rho_V,m_V)\,dx\, dt\\
        & + \int^T_0 \mathcal{G}_P((\rho_S + \rho_I + \rho_R)(t,\cdot)) + \mathcal{G}_V(\rho_V(t,\cdot)) \, dt\\
        & + \int^{T'}_0 \mathcal{G}_0(f(t,\cdot))\, dt\\
        & + \frac{\lambda}{2}\int^T_0\int_\Omega f^2 + \sum_{i\in\mathbb{S}} \rho_i^2 + |m_i|^2 \,dx\,dt.
    \end{split}
\end{equation}
In the perspective of a control problem, the first term at the right-hand side in~\eqref{eq:var_vacc} is the terminal cost, while the rest of the terms accounts for the running costs. The quadratic terms in the last line is a $\lambda$-strongly convex functional. The functional $F$ is  $\lambda$-strongly convex if for any $u=((\rho_i,m_i)_{i\in\mathbb{S}},f)$, $F$ satisfies
\[
    F(\tilde u) \geq F(u) + \partial F(u)(\tilde u-u) + \frac{\lambda}{2} \|\tilde u-u\|_{L^2}^2, \quad \text{for all } \tilde u=((\tilde\rho_i,\tilde m_i)_{i\in\mathbb{S}},\tilde f)
\]
where $\|\tilde u - u\|_{L^2}^2$ is defined as
\[
    \int^T_0 \int_\Omega (\tilde f - f)^2 + \sum_{i\in\mathbb{S}}(\tilde \rho_i - \rho_i)^2+|\tilde m_i - m_i|^2\,dx\,dt
\]
and $\partial F$ denotes the convex subdifferential of $F$. Since $\mathcal{E}_i$, $F_i$, $\mathcal{G}_i$ are convex and lower-semicontinuous, $G$ is  $\lambda$-strongly convex as the sum of convex and  $\lambda$-strongly convex functionals. The strong convexity of $G$ is important as the algorithm of the paper requires the objective cost functional to be strongly convex (Theorem~\ref{thm:local-convergence}).

\subsection{Constraints for vaccine production}\label{subsection:constraints-vaccine-production}

In addition to the constraint from \eqref{eq:continuity_vacc},  we adapt the following constraints to reflect the limited vaccination coverage:
 \begin{equation}\label{eq:constraint2}
    \begin{aligned}
        &0 \leq f(t,x) \leq f_{max} && (t,x) \in [0,T'] \times \Omega_{factory}\\    
        &f(t,x) = 0 && (t,x) \in [0,T'] \times \Omega \backslash \Omega_{factory}\\
        & \rho_V(t,x) \leq C_{factory} && (t,x) \in [0,T'] \times \Omega_{factory}
    \end{aligned}
 \end{equation}
 where $\Omega_{factory} \subset \Omega$ indicates the factory area where vaccines are produced and $f_{max}$ is a nonnegative constant representing the maximum vaccine production rate. In the third inequality, a nonnegative constant $C_{factory}$ limits the total number of vaccines produced  during $0 < T < T'$.
 \[
    \int_0^{T'} \int_{\Omega} \rho_V(t,x)\,dx\,dt \leq C_{factory} T' |\Omega_{factory}|.
 \]
 The constraints~\eqref{eq:constraint2} can be imposed by having the following functionals for $\mathcal{G}_V$ and $\mathcal{G}_0$.
 \begin{equation}\label{eq:constraint-V-0}
     \begin{aligned}
         \mathcal{G}_V(\rho_V(t,\cdot))& = \int_\Omega g_V(\rho_V(t,x))\, dx + i_{[-\infty, C_{factory})}(\rho_V(t,\cdot))\\
         \mathcal{G}_0(f(t,\cdot)) &= \int_\Omega g_0 (f(t,x)) + i_{\Omega_{factory}}(x)f(t,x) \, dx + i_{[-\infty, f_{max})}(f(t,\cdot))
     \end{aligned}
 \end{equation}
where $\Omega_{factory} \subset \Omega$ indicates the factory area where vaccines are produced. The functionals $i_{[-\infty, C_{factory}]}$ and $i_{[-\infty, f_{max}]}$ are defined as
\begin{equation*}
    \begin{aligned}
        i_{[a,b]} (u)=
        \begin{cases}
          0, & a \leq u(x) \leq b\quad \text{for all }x\in\Omega\\
          \infty, & \text{otherwise}
        \end{cases}
    \end{aligned}
\end{equation*}
where $a,b$ are constants and $u:\Omega\rightarrow\mathbb{R}$ is a function. The function $i_{\Omega_{factory}}(x)$ is defined as
\begin{equation*}
    i_{\Omega_{factory}}(x) = 
    \begin{cases}
        0, & x \in \Omega_{factory}\\
        \infty, & x \in \Omega \backslash \Omega_{factory}.
    \end{cases}
\end{equation*}
This function forces $f(t,x)=0$ if $(t,x)\in(0,T')\times(\Omega\backslash\Omega_{factory})$, thus vaccines are produced only in $\Omega_{factory}$. 

\begin{remark}
    The formulation is not limited to SIR epidemic model. For example, we can describe the SIRD (Susceptible-Infected-Recovered-Deceased) epidemic model by adding an extra population $\rho_D$ for the deceased population with a mortality rate $\mu$.
    \begin{equation*}
       \begin{cases}
          \partial_t \rho_S + \nabla \cdot m_S = - \beta \rho_S K * \rho_I + \frac{\eta_S^2}{2} \Delta \rho_S - \theta_1 \rho_V \rho_S & (t,x)\in(0,T)\times\Omega\\
        \partial_t \rho_I + \nabla \cdot m_I = \beta \rho_S K * \rho_I - \gamma \rho_I - \mu \rho_I + \frac{\eta^2_I}{2} \Delta \rho_I  & (t,x) \in(0,T)\times\Omega\\
        \partial_t \rho_R + \nabla \cdot m_R = \gamma \rho_I + \frac{\eta^2_R}{2} \Delta \rho_R + \theta_1 \rho_V \rho_S  & (t,x) \in(0,T)\times\Omega\\
        \partial_t \rho_D = \mu \rho_I + \frac{\eta^2_D}{2} \Delta \rho_D   & (t,x) \in(0,T)\times\Omega\\
        \partial_t \rho_V = f(t,x) - \theta_2 \rho_V \rho_S & (t,x) \in(0,T')\times\Omega\\
        \partial_t \rho_V + \nabla \cdot m_V  = - \theta_2 \rho_V \rho_S  & (t,x) \in[T',T)\times\Omega\\
         \rho_S(0,\cdot), \rho_I(0,\cdot), \rho_R(0,\cdot), \rho_D(0,\cdot), \rho_V(0,\cdot) \text{ are given.}
        \end{cases}
    \end{equation*}
\end{remark}

 \subsection{Properties}
From the definition of the cost functional and the constraint~\eqref{eq:continuity_vacc}, we have the following minimization problem:
\begingroup
\allowdisplaybreaks
\begin{align}\label{eq:SIRV-variational-problem}
        &\inf_{(\rho_i,m_i)_{i\in\mathbb{S}},f} \Big\{ {G}((\rho_i,m_i)_{i\in\mathbb{S}}, f): \text{subject to~\eqref{eq:continuity_vacc}}\Big\}.
\end{align}
\endgroup
We first define the inner product of vectors of functions in $L^2$. Given vectors of functions $u=(u_1(t,x),u_2(t,x),\cdots,u_k(t,x))$ and $v=(v_1(t,x),v_2(t,x),\cdots,v_k(t,x))$ with $u_i,v_i: [0,T]\times\Omega\rightarrow \mathbb{R}$, the $L^2$ inner product of vectors $u$ and $v$ and $L^2$ norm of $u$ are defined by 
\begin{equation}\label{eq:definition-of-L2-vector-inner}
    \langle u, v \rangle_{L^2} =  \sum^k_{i=0} (u_i, v_i)_{L^2},\quad \|u\|_{L^2}^2 = \langle u, u \rangle_{L^2} 
\end{equation}
where $(\cdot,\cdot)_{L^2([0,T]\times\Omega)}$ is a $L^2$ inner product such that
\[
    (u,v)_{L^2([0,T]\times\Omega)} = \int^T_0 \int_\Omega u(t,x) v(t,x) \, dx \, dt.
\]
We introduce dual variables $(\phi_i)_{i\in\mathbb{S}}$ for each continuity equation from~\eqref{eq:nonlinear-operator-A}.  Using the dual variables and the definitions of the inner products, we convert the minimization problem into a saddle point problem.
\begin{equation}\label{eq:saddle}
    \inf_{(\rho_i,m_i)_{i\in\mathbb{S}},f} \sup_{(\phi_i)_{i\in\mathbb{S}}}\, \mathcal{L}((\rho_i,m_i,\phi_i)_{i\in\mathbb{S}},f),
\end{equation}
where $\mathcal{L}$ is the Lagrangian functional defined as
\begingroup
\allowdisplaybreaks
\begin{align*}
        &\mathcal{L}((\rho_i,m_i,\phi_i)_{i\in\mathbb{S}},f)\\ 
        &=  {G} ((\rho_i,m_i)_{i\in\mathbb{S}},f) -\langle A((\rho_i,m_i)_{i\in\mathbb{S}},f) , (\phi_i)_{i\in\mathbb{S}} \rangle_{L^2}\\
        &=  {G} ((\rho_i,m_i)_{i\in\mathbb{S}},f)\nonumber\\
        &\quad - \int^T_0 \int_\Omega \phi_S \left( \partial_t \rho_S + \nabla \cdot m_S + \beta \rho_S K * \rho_I + \theta_1 \rho_S \rho_V - \frac{\eta_S^2}{2} \Delta \rho_S \right) \, dx \, dt\nonumber\\
        &\quad - \int^T_0 \int_\Omega \phi_I \left( \partial_t \rho_I + \nabla \cdot m_I - \beta \rho_S K * \rho_I + \gamma \rho_I - \frac{\eta_I^2}{2} \Delta \rho_I \right) \, dx \, dt\nonumber\\
        &\quad - \int^T_0 \int_\Omega \phi_R \left( \partial_t \rho_R + \nabla \cdot m_R - \gamma \rho_I - \theta_1 \rho_S \rho_V - \frac{\eta_R^2}{2} \Delta \rho_R \right) \, dx \, dt\nonumber\\
        &\quad - \int^{T}_0 \int_\Omega \phi_V \left( \partial_t \rho_V - f \mathcal{X}_{[0,T')}(t) + \nabla \cdot m_V \mathcal{X}_{[T',T]}(t) + \theta_2 \rho_S \rho_V  \right) \, dx \, dt.
\end{align*}
\endgroup
For brevity, we denote
\[
    u = ((\rho_i,m_i)_{i\in\mathbb{S}},f),\quad p = (\phi_i)_{i\in\mathbb{S}}.
\]
We can rewrite the Lagrangian as
\begin{equation}\label{eq:original-Lagrangian}
    \mathcal{L}(u,p) = {G}(u) - \langle A(u), p \rangle_{L^2}
\end{equation}
where the nonlinear operator $A(u)$ is defined as
\begin{equation}\label{eq:def-of-nonlinear-op-A}
        A(u) = ( A_S(u), A_I(u), A_R(u), A_V(u) )
\end{equation}
\begin{equation*}
    \begin{aligned}
    A_S(u) &= \partial_t \rho_S + \nabla \cdot m_S - \frac{\eta_S^2}{2} \Delta \rho_S + \beta\rho_S K * \rho_I + \theta_1 \rho_S \rho_V,\\
    A_I(u) &= \partial_t \rho_I + \nabla \cdot m_I - \frac{\eta_I^2}{2} \Delta \rho_I - \beta\rho_I K * \rho_S + \gamma \rho_I,\\
    A_R(u) &= \partial_t \rho_R + \nabla \cdot m_R - \frac{\eta_R^2}{2} \Delta \rho_R - \gamma \rho_I,\\
    A_V(u) &= \partial_t \rho_V - f \mathcal{X}_{[0,T')}(t) + \nabla \cdot m_V \mathcal{X}_{[T',T]}(t) + \theta_1 \rho_S \rho_V.
    \end{aligned}
\end{equation*}
As noted in \cite{lee2020controlling}, the dual gap, the difference between the primal solution and dual solution, may not be zero because the nonconvex functions $(\rho_S,\rho_I) \mapsto \rho_S K * \rho_I$ and $(\rho_S,\rho_V) \mapsto \rho_S \rho_V$ make the feasible set nonconvex. We circumvent the problem by linearizing the nonlinear operator at a base point~$\bar{u}$
\[
    A(u) \approx \bar{A}_{\bar u}(u) = A(\bar{u}) + [\nabla A(\bar{u})](u-\bar{u}).
\]
In our formulation, the linearlized operator $\bar A_{\bar u}(u)$ can be written as follows.
\begin{equation*}
    \begin{aligned}
    \bar A_{\bar u}(u) &= ( \bar A_{S \bar u}(u), \bar A_{I \bar u}(u ), \bar A_{R\bar u}(u), \bar A_{V\bar u}(u) )\\
    \bar A_{S \bar u}(u) &= \partial_t \rho_S + \nabla \cdot m_S - \frac{\eta_S^2}{2} \Delta \rho_S + \beta\rho_S K * \bar \rho_I + \theta_1 \rho_S \bar \rho_V,\\
    \bar A_{I \bar u}(u) &= \partial_t \rho_I + \nabla \cdot m_I - \frac{\eta_I^2}{2} \Delta \rho_I - \beta\rho_I K * \bar \rho_S + \gamma \rho_I,\\
    \bar A_{R \bar u}(u) &= \partial_t \rho_R + \nabla \cdot m_R - \frac{\eta_R^2}{2} \Delta \rho_R  - \gamma \bar \rho_I,\\
    \bar A_{V \bar u}(u) &= \partial_t \rho_V - f \mathcal{X}_{[0,T')}(t) + \nabla \cdot m_V \mathcal{X}_{[T',T]}(t) + \theta_1 \rho_V \bar \rho_S
    \end{aligned}
\end{equation*}
where $\bar u = u = ((\bar \rho_i,\bar m_i)_{i\in\mathbb{S}},\bar f)$.
We define a linearized Lagrangian as
\begin{equation}\label{eq:linearized-Lagrangian}
    \bar{\mathcal{L}}_{\bar u}(u,p) = {G}(u) - \langle \bar{A}_{\bar u}(u), p \rangle_{L^2}.  
\end{equation}
In the paper~\cite{clason2017primal}, the author developed a primal-dual algorithm using the linearized Lagrangian (Algorithm~\eqref{eq:alg-nonlinear-PDHG}) and proves that the sequence $(u^{(k)},p^{(k)})^\infty_{k=1}$ from the algorithm converges to the saddle point $(u_*,p_*)$ (in Section~\ref{section:algorithms-subsection:properties}, we prove the local convergence to the saddle point given $(u^{(0)},p^{(0)})$ is sufficiently close to the saddle point). By the first-order optimality conditions (also known as Karush-Kuhn-Tucker (KKT) conditions), the saddle point satisfies
\begin{equation}\label{eq:linearized-KKT}
\begin{split}
    [\nabla A(u_*)]^T p_* &\in \partial {G}(u_*)\\
    A(u_*) &= 0.
\end{split}
\end{equation}
\noindent In the next proposition, we present the equations derived from the KKT conditions~\eqref{eq:linearized-KKT}. 
\begin{proposition}[Mean-field control SIRV system]\label{proposition:KKT-conditions}
By KKT conditions, the saddle point $((\rho_i, m_i, \phi_i)_{i\in\mathbb{S}},f)$ of
\eqref{eq:saddle} satisfies the following equations.
    \begingroup
    \allowdisplaybreaks
    \begin{align*}\label{MFSIR}
            &\partial_t \phi_S - \frac{\alpha_S}{2} |\nabla \phi_S|^2 + \frac{\eta^2_S}{2}\Delta \phi_S + \frac{\delta \mathcal{G}_P}{\delta \rho}(\rho_S+\rho_I+\rho_R) + \beta (\phi_I - \phi_S) K * \rho_I\\
            &\hspace{4.5cm} + \rho_V\bigl( \theta_1 (\phi_R - \phi_S) - \theta_2 \phi_V)\bigl)= 0 && (t,x)\in (0,T)\times \Omega\\
            &\partial_t \phi_I - \frac{\alpha_I}{2} |\nabla \phi_I|^2 + \frac{\eta^2_I}{2}\Delta \phi_I +  \frac{\delta \mathcal{G}_P}{\delta \rho} (\rho_S+\rho_I+\rho_R)\\
            &\hspace{4.5cm} + \beta K * \left( \rho_S (\phi_I - \phi_S) \right) + \gamma (\phi_R - \phi_I) = 0 && (t,x)\in (0,T)\times \Omega\\
            &\partial_t \phi_R - \frac{\alpha_R}{2} |\nabla \phi_R|^2 + \frac{\eta^2_R}{2}\Delta \phi_R + \frac{\delta \mathcal{G}_P}{\delta \rho} (\rho_S+\rho_I+\rho_R)= 0  && (t,x)\in (0,T)\times \Omega\\
            &\partial_t \phi_V + \frac{\delta \mathcal{G}_V}{\delta \rho} (\rho_V) + \rho_S \bigl( \theta_1 (\phi_R - \phi_S) - \theta_2 \phi_V)\bigl)= 0 && (t,x)\in (0,T')\times \Omega\\
            &\partial_t \phi_V - \frac{\alpha_V}{2}|\nabla\phi_V|^2  + \frac{\delta \mathcal{G}_V}{\delta \rho} (\rho_V)  + \rho_S \bigl( \theta_1 (\phi_R - \phi_S) - \theta_2 \phi_V)\bigl)= 0 && (t,x)\in (T',T)\times \Omega\\
            & \partial_t \rho_S -\frac{1}{\alpha_S}\nabla \cdot (\rho_S\nabla\phi_S) + \beta \rho_S K * \rho_I + \theta_1 \rho_S \rho_V - \frac{\eta_S^2}{2} \Delta\rho_S = 0 && (t,x)\in (0,T)\times \Omega\\
            &\partial_t \rho_I -\frac{1}{\alpha_I}\nabla \cdot(\rho_I \nabla\phi_I) - \beta \rho_S K *\rho_I + \gamma \rho_I - \frac{\eta^2_I}{2} \Delta\rho_I = 0 && (t,x)\in (0,T)\times \Omega\\
            & \partial_t \rho_R -\frac{1}{\alpha_R} \nabla \cdot (\rho_R \nabla\phi_R) - \gamma \rho_I - \theta_1 \rho_S \rho_V - \frac{\eta^2_R}{2} \Delta \rho_R= 0 && (t,x)\in (0,T)\times \Omega\\
            & \partial_t \rho_V - f + \theta_2 \rho_S \rho_V = 0 && (t,x)\in (0,T')\times \Omega\\
            & \partial_t \rho_V -\frac{1}{\alpha_V} \nabla \cdot (\rho_V \nabla\phi_V)  + \theta_2 \rho_S \rho_V = 0 && (t,x)\in (T',T)\times \Omega\\
            & \frac{\delta \mathcal{G}_0}{\delta f}(f) + \phi_V = 0  && (t,x)\in (0,T')\times \Omega\\
            &\phi_i(T,\cdot) = \frac{\delta \mathcal{E}_i}{\delta \rho(T,\cdot)} (\rho_i(T,\cdot)), \quad i \in \mathbb{S}.
    \end{align*}
    \endgroup
    The terms $\frac{\delta \mathcal{G}_P}{\delta \rho}$, $\frac{\delta \mathcal{G}_V}{\delta \rho}$, $\frac{\delta \mathcal{G}_P}{\delta \rho}$, $\frac{\delta \mathcal{G}_0}{\delta f}$, and $\frac{\delta \mathcal{E}_i}{\delta \rho(T,\cdot)}$ are the functional derivatives. In other words, given $F:\mathcal{H}\rightarrow\mathbb{R}$ be a smooth functional where $\mathcal{H}$ is a separable Hilbert space and $\rho\in\mathcal{H}$, we say a map $\frac{\delta F}{\delta \rho}$ is the functional derivative of $F$ with respect to $\rho$ if it satisfies
    \[
        \lim_{\epsilon\rightarrow 0} \frac{F(\rho+\epsilon h) - F(\rho)}{\epsilon} = \int_\Omega \frac{\delta F}{\delta \rho}(\rho(x)) h(x)\, dx,
    \]
    for any arbitrary function $h:\Omega\rightarrow \mathbb{R}$.
\end{proposition}
The dynamical system models the optimal vector field strategies for S, I, R populations and the vaccine distribution. It combines both strategies from mean field controls and SIRV models. For this reason, we call it \textit{Mean-field control SIRV system}.
The proof of Proposition~\ref{proposition:KKT-conditions} can be found in the Appendix.

\section{Algorithms}\label{section:algorithms}


In this section, we propose an algorithm to solve the proposed SIRV variational problem. We use the primal-dual hybrid gradient (PDHG) algorithm~\cite{champock11,champock16}. The PDHG can solve the following convex optimization problem.
\[
    \begin{aligned}
        \min_u\, f(Au) + g(u)
    \end{aligned}
\]
where $f$ and $g$ are convex functions and $A$ is a continuous linear operator. The algorithm solves the problem by converting the problem into a saddle point problem by introducing a dual variable $p$.
\[
    \begin{aligned}
        \min_u\,\max_p\, g(u) + \langle Au, p \rangle_{L^2} - f^*(p)
    \end{aligned}
\]
with $L^2$ inner product is defined in~\eqref{eq:definition-of-L2-vector-inner} and
\[
    f^*(p) = \sup_{u} \, \langle u, p \rangle_{L^2} - f(u)
\] 
is the Legendre transform of $f$. The method solves the saddle point problem by iterating
\begin{equation}\label{eq:alg-PDHG}
    \begin{aligned}
        u^{(k+1)} &= \argmin_u \, g(u) + \langle u, A^T p^{(k)} \rangle_{L^2} + \frac{1}{2\tau} \|u-u^{(k)}\|^2_{L^2}\\
        \tilde{u}^{(k+1)} &= 2 u^{(k+1)} - u^{(k)}\\
        p^{(k+1)} &= \argmax_p \, \langle A \tilde{u}^{(k+1)}, p \rangle_{L^2} - f^*(p) - \frac{1}{2\sigma} \|p-p^{(k)}\|^2_{L^2}.
    \end{aligned}
\end{equation}
The scheme converges if the step sizes $\tau$ and $\sigma$ satisfy
\begin{equation} \label{eq:step-sizes-condition}
    \tau \sigma \|A^T A\|_{L^2} < 1,
\end{equation}
where $\|\cdot\|$ is an operator norm in $L^2$. However, the SIRV variational problem has a nonlinear function $A$ for the constraint. Thus, we use the extension of the algorithm from~\cite{clason2017primal} which solves the nonlinear constrained optimization problem.
\begin{equation}\label{eq:nonlinear-saddle-point}
    \min_u\,\max_p\, g(u) + \langle A(u), p \rangle_{L^2} - f^*(p),
\end{equation}
where $A$ is a nonlinear function. The scheme iterates the algorithm~\eqref{eq:alg-PDHG} with a linear approximation of $A$ at a base point $\bar u$
\[
    A(u) \approx A(\bar u) + [\nabla A(\bar u)](u - \bar u).
\]
Denote $A_{u} := \nabla A(u)$. We have a linearized saddle point problem
\begin{equation}\label{eq:linearlized-saddle-point}
    \min_u\,\max_p\, g(u) + \langle A(\bar u) + A_{\bar u}(u-\bar u), p \rangle_{L^2} - f^*(p)
\end{equation}
and the scheme iterates
\begin{equation}\label{eq:alg-nonlinear-PDHG}
    \begin{split}
        u^{(k+1)} &= \argmin_u \, g(u) + \langle u, A_{u^{(k)}}^T p^{(k)} \rangle_{L^2} + \frac{1}{2\tau^{(k)}} \|u-u^{(k)}\|^2_{L^2}\\
        \tilde{u}^{(k+1)} &= 2 u^{(k+1)} - u^{(k)}\\
        p^{(k+1)} &= \argmax_p \, \langle A(u^{(k)}) + A_{u^{(k)}} (\tilde{u}^{(k+1)} - u^{(k)}) , p \rangle_{L^2} - f^*(p) - \frac{1}{2\sigma^{(k)}} \|p-p^{(k)}\|^2_{L^2}
    \end{split}
\end{equation}
The paper~\cite{clason2017primal} proves that the sequence $\{u^{(k)},p^{(k)}\}^\infty_{k=0}$ of the algorithm converges to some saddle point $(u_*,p_*)$ that satisfies~\eqref{eq:linearized-KKT}. However, the scheme converges if the step sizes satisfy
\[
    \sigma^{(k)} \tau^{(k)} \|\nabla A(u^{(k)})\|^2_{L^2} < 1,\quad k=1,2,\cdots.
\]
Suppose we use an unbounded operator that depends on the grid size, for example, $A = \nabla$. The discrete approximation of the operator norm of $A$ increases as the grid size increases (Figure~\ref{fig:op-norm-grid-size} illustrates the relationship between the norm of an unbounded operator and grid sizes). Thus, the scheme can result in a very slow convergence if we use a fine grid resolution. To circumvent the problem, we use the General-proximal Primal-Dual Hybrid Gradient (G-prox PDHG) method from \cite{JacobsLegerLiOsher2018_solvinga} which is another variation of the PDHG algorithm. This variant provides an appropriate choice of norms for the algorithm, and the authors prove that choosing the proper norms allows the algorithm to have larger step sizes than the vanilla PDHG algorithm. The G-prox PDHG iterates
\begin{equation}\label{eq:alg-G-prox-PDHG}
    \begin{aligned}
        u^{(k+1)} &= \argmin_u \, g(u) + \langle u, A_{u^{(k)}}^T p^{(k)} \rangle_{L^2} + \frac{1}{2\tau^{(k)}} \|u-u^{(k)}\|^2_{L^2}\\
        \tilde{u}^{(k+1)} &= 2 u^{(k+1)} - u^{(k)}\\
        p^{(k+1)} &= \argmax_p \, \langle A(u^{(k)}) + A_{u^{(k)}} (\tilde{u}^{(k+1)} - u^{(k)}) , p \rangle_{L^2} - f^*(p) - \frac{1}{2\sigma^{(k)}} \|p-p^{(k)}\|^2_{\mathcal{H}^{(k)}}.
    \end{aligned}
\end{equation}
where the norm $\|\cdot\|_{\mathcal{H}^{(k)}}$ is defined as
\[
    \|p\|^2_{\mathcal{H}^{(k)}} = \|A_{u^{(k)}}^T p\|^2_{L^2}.
\]
By choosing the proper norms, the step sizes only need to satisfy
\[
    \sigma^{(k)} \tau^{(k)} < 1,\quad k=1,2,\cdots
\]
which are clearly independent of the grid size.

\begin{figure}[h]
    \centering
    \begin{subfigure}[b]{0.3\textwidth}
         \centering
         \includegraphics[width=1\linewidth]{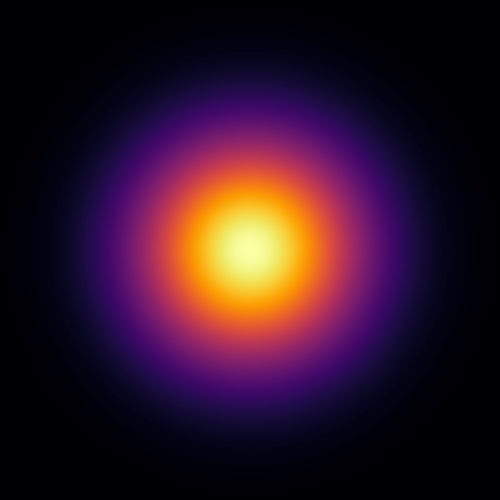}
         \caption{\small $u(x) = e^{-20 |x|^2}$}
     \end{subfigure}
     \hspace{0.5cm}
     \begin{subfigure}[b]{0.5\textwidth}
         \centering
         \includegraphics[width=1\linewidth]{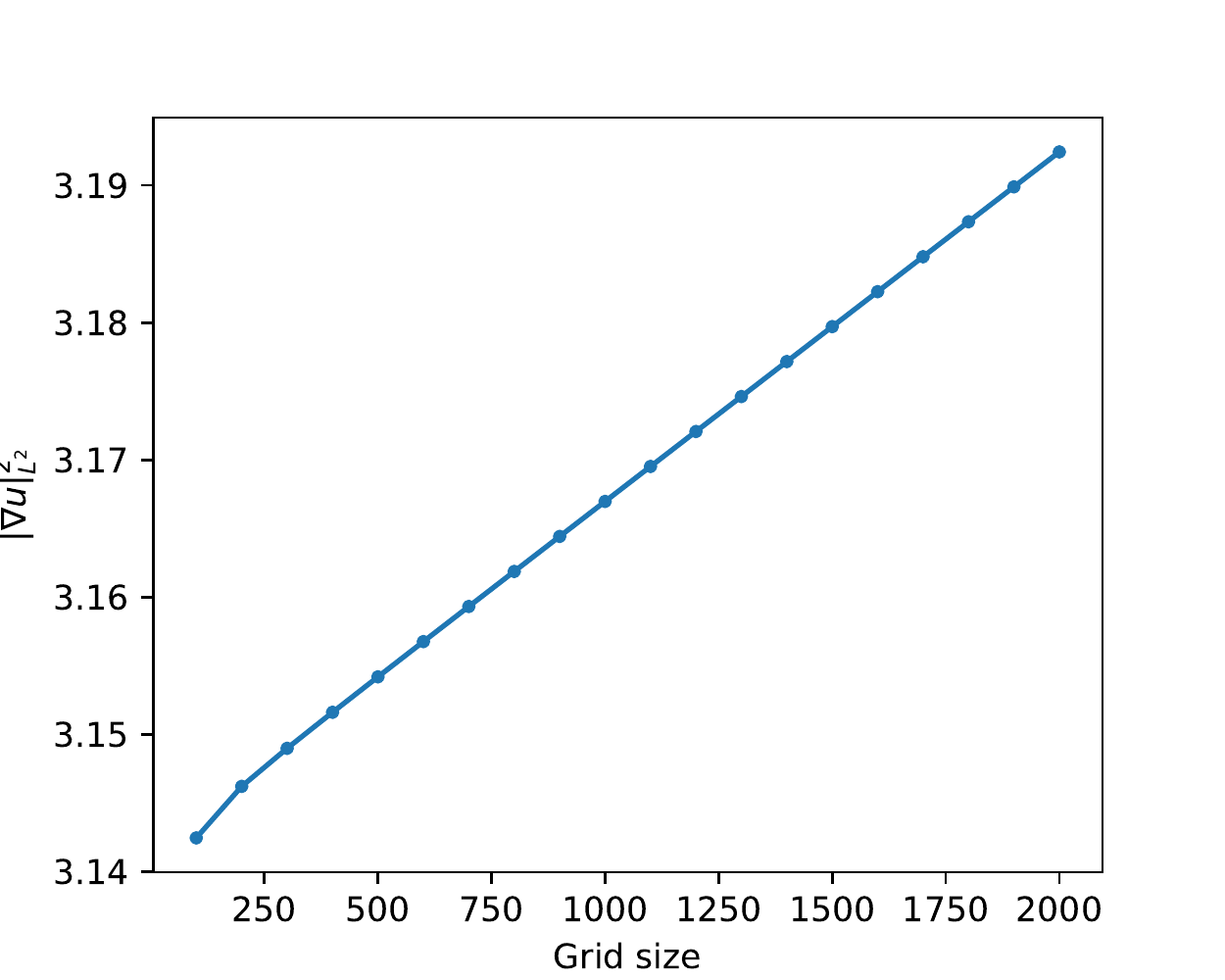}
         \caption{\small Operator norm vs. grid sizes}
     \end{subfigure}
     \hfill
    \caption{\small The image (a) shows $u$ on a unit square domain $[-0.5,0.5]^2$. The plot~(b) shows the discrete approximation of $\|\nabla u\|^2_{L^2} = \int_\Omega |\nabla u(x)|^2 \,dx$ with respect to grid sizes. It shows that in the discrete approximation, the norm of an unbounded operator $\nabla$ increases as grid size increases. }
    \label{fig:op-norm-grid-size}
\end{figure}

\subsection{Local convergence of the algorithm}\label{section:algorithms-subsection:properties}
In this section, we show the iterations from the algorithm~\eqref{eq:alg-G-prox-PDHG} locally converges to the saddle point. The local convergence theorem in this paper is mainly based on the Theorem~2.11 from~\cite{clason2017primal}. However, we add a preconditioning operator from the G-prox PDHG method. We show that the method converges locally to the saddle point with the step sizes independent of the nonlinear operator $A$.

From the algorithm~\eqref{eq:alg-G-prox-PDHG}, $(u^{(k+1)},p^{(k+1)})$ satisfies the following first-order optimality conditions
\begin{equation}
    \begin{split}
        0 &\in \partial g(u^{(k+1)}) + A_{u^{(k)}}^T p^{(k)} + \frac{1}{\tau^{(k)}} (u^{(k+1)} - u^{(k)})\\
        0 &\in A(u^{(k)}) + 2 A_{u^{(k)}} (u^{(k+1)} - u^{(k)}) - \partial f^* (p^{(k+1)}) - \frac{1}{\sigma^{(k)}} A_{u^{(k)}} A_{u^{(k)}}^T (p^{(k+1)}-p^{(k)})
    \end{split}
\end{equation}
which can be rewritten as
\begin{equation}
    \begin{split}
        0 \in H_{u^{(k)}} (q^{(k+1)}) + M^{(k)} (q^{(k+1)} - q^{(k)})
    \end{split}
\end{equation}
with $q=(u,p)$. Here, the monotone operator $H_{\bar u}$ is defined as
\[
    H_{\bar u} (q) := 
    \begin{pmatrix}
    \partial g(u) + A_{\bar u }^T p\\
    \partial f^*(p) - A(\bar u) - A_{\bar u} (u - \bar u)
    \end{pmatrix}
\]
and
\[
    M^{(k)} :=
    \begin{pmatrix}
        \frac{1}{\tau^{(k)}} Id & - A_{u^{(k)}}^T\\
        - A_{u^{(k)}} & \frac{1}{\sigma^{(k)}} A_{u^{(k)}} A_{u^{(k)}}^T
    \end{pmatrix}
\]
where $Id$ is an identity operator.

Recall that from~\eqref{eq:linearized-KKT}, the saddle point $q_*=(u_*,p_*)$ has to satisfy
\[
    0 \in H_{u_*}(u_*,p_*).
\]
Throughout, we assume that
\begin{equation}\label{ass:nabla-A-greater-than-0}
    \|\nabla A(u_*)\| > 0\quad\text{and}\quad\text{$u\mapsto A(u)$ is continuous}.
\end{equation}

\begin{lemma}\label{lemma:nabla-A-bounded}
    There exists constants $0<c<C$ and $R>0$ such that
    \[
        c \leq \|\nabla A(u)\| \leq C,\quad (\|u-u_*\|_{L^2} \leq R)
    \]
    where $\|\cdot\|$ is an operator norm.
\end{lemma}
\begin{proof}
    This follows immediately from~\eqref{ass:nabla-A-greater-than-0} and the fact that the derivative $\nabla A(u)$ is continuous with respect to $u$.
\end{proof}

\begin{lemma}\label{lemma:M-k-bounded}
    Suppose~\eqref{ass:nabla-A-greater-than-0} holds and let $\tau^{(k)} \sigma^{(k)} <1$. Then there exist constants $0<\theta < \Theta$ such that
    \[
        \theta^2 \|q\|_{L^2}^2 \leq \langle q, M^{(k)} q \rangle \leq \Theta^2 \|q\|_{L^2}^2
    \]
    where
    \[
    \|q\|_{L^2}^2 = \|u\|_{L^2}^2 + \|p\|_{L^2}^2.
    \]
\end{lemma}
A proof of Lemma~\ref{lemma:M-k-bounded} is provided in the appendix.

With the above Lemmas, we can use the Theorem~2.11 from~\cite{clason2017primal} to show the local convergence of the algorithm.
\begin{theorem}\label{thm:local-convergence}
    Let $(u_*, p_*) \in L^2 \times \mathcal{H}^{(*)}$ be a solution to~\eqref{eq:linearized-KKT} where $\|p\|^2_{\mathcal{H}^{(*)}} = \|A_{u_*}^T p\|^2_{L^2}$. Let the step sizes $\tau^{(k)}$ and $\sigma^{(k)}$ satisfy $\tau^{(k)}\sigma^{(k)}<1$ for all $k$. Then there exists $\delta>0$ such that for any initial point $(u^{(0)},p^{(0)})\in L^2 \times \mathcal{H}^{(0)}$ satisfying
    \[
        \|u^{(0)} - u_*\|_{L^2}^2 + \|p^{(0)} - p_*\|_{L^2}^2 < \delta^2,
    \]
    the iterates $(u^{(k)}, p^{(k)})$ from~\eqref{eq:alg-G-prox-PDHG} converges to the saddle point $(u_*, p_*)$.
\end{theorem}
\begin{proof}
    By Lemma~\ref{lemma:nabla-A-bounded}, Lemma~\ref{lemma:M-k-bounded}, and strong convexity of the functional $G$ from~\eqref{eq:var_vacc}, we can use~\cite[Theorem~2.11]{clason2017primal}, which proves the theorem.
\end{proof}

\begin{remark}
    \cite[Theorem~2.11]{clason2017primal} requires $H_{u_*}$ to satisfy the condition called metric regularity. In our formulation, the constraint $A(u) = 0$ makes $H_{u_*}$ metrically regular by \cite[Section~5.3]{clason2017stability}. We refer readers to~\cite{clason2017primal, clason2017stability, rockafellar2009variational} for further details about metric regularity.
\end{remark}

\subsection{Implementation of the algorithm}

To implement the algorithm to the minimization problem \eqref{eq:var_vacc}, we set
\begin{equation*}
    \begin{split}
        u &= ((\rho_i, m_i)_{i\in\mathbb{S}},f)\\
        p &= (\phi_i)_{i\in\mathbb{S}}\\
        g(u) &= {G} (u)\\    
        f(A(u)) &= 
    \begin{cases}
     0 & \text{if } A(u) = 0\\
     \infty & \text{otherwise}
    \end{cases}\\
    f^*(p) & = 0.
    \end{split}
\end{equation*}
We use~\eqref{eq:def-of-nonlinear-op-A} for the definition of the operator $A$.
Define the Lagrangian functional as
\[
    \mathcal{L}(u,p) := {G}(u) - \langle A(u), p \rangle_{L^2}
\]
where $\langle \cdot, \cdot \rangle_{L^2}$ is defined in~\eqref{eq:definition-of-L2-vector-inner}. We summarize the algorithm as follows. 


\begin{algorithm}[H]
\caption{G-prox PDHG for mean-field control SIRV system} \label{alg:gproxPDHG}
    \medskip
    
    \textbf{Input}: $\rho_i(0,\cdot)$ ($i\in \mathbb{S}$)
    
    \textbf{Output}: $\rho_i, m_i, \phi_i$ ($i\in \mathbb{S}$), $f$


    \rule{\linewidth}{0.5pt}

        \textbf{While} relative error $>$ tolerance \textbf{For} $i\in\mathbb{S}$
    
            \begin{align*}
                \rho_i^{(k+1)} &= \argmin_\rho \mathcal{L}((\rho, m^{(k)},f^{(k)}),\phi^{(k)}) + \frac{1}{2\tau} \|\rho - \rho_i^{(k)}\|^2_{L^2}\\
                m_i^{(k+1)} &= \argmin_m \mathcal{L}((\rho^{(k+1)}, m,f^{(k)}),\phi_i^{(k)}) + \frac{1}{2\tau} \|m - m_i^{(k)}\|^2_{L^2}\\
                f^{(k+1)} &= \argmin_f \mathcal{L}((\rho^{(k+1)},m^{(k+1)},f),\phi^{(k)}) + \frac{1}{2\tau} \|f - f^{(k)}\|^2_{L^2}\\
                \phi_i^{(k+ 1)} &= \argmax_\phi \mathcal{L}((2\rho^{(k+1)} - \rho^{(k)}, 2m^{(k+1)} - m^{(k)},2f^{(k+1)} -f^{(k)}),\phi)\\
                & \hspace{1.8cm} - \frac{1}{2\sigma} \|\phi - \phi_i^{(k)}\|^2_{H^{(k)}_i}
            \end{align*}
\end{algorithm}

\noindent Here, $L^2$ and $H^{(k)}_i$ norms are defined as
\begin{align*}
        \|u\|^2_{L^2} &= (u,u)_{L^2}= \int^T_0 \int_\Omega u^2 dx\, dt,\quad \|p\|^2_{H^{(k)}_i} = \|[\nabla A_i(u^{(k)})]^T p\|_{L^2}^2,\quad i \in \mathbb{S}
\end{align*}
for any $u : [0,T] \times \Omega \rightarrow [0,\infty)$. Moreover, the relative error is defined as
\[
\text{relative error} = \frac{|{G}(\rho_i^{(k+1)},m_i^{(k+1)})- {G}(\rho_i^{(k)},m_i^{(k)})|}{|{G}(\rho_i^{(k)},m_i^{(k)})|}.
\]
In the section~\ref{sec:experiments}, We use quadratic functions for $\mathcal{E}_i$ $(i\in\{S,I,V\})$, $\mathcal{G}_P$,   $\mathcal{G}_V$, $\mathcal{G}_0$. With the definitions~\eqref{eq:constraint-V-0}, we use
\[
    \begin{aligned}
        \mathcal{E}_i(\rho_i(T,\cdot)) &= \int_\Omega \frac{a_i}{2} \rho_i(T,x)^2 \, dx, \quad i=S,I,V\\
        \mathcal{G}_P(\rho(t,\cdot)) &= \int_\Omega \frac{d_P}{2} \rho(t,x)^2 \, dx\\
        \mathcal{G}_V(\rho(t,\cdot))& = \int_\Omega \frac{d_V}{2} \rho(t,x)^2\, dx + i_{[-\infty , C_{factory}]}(\rho(t,\cdot))\\
        \mathcal{G}_0(f(t,\cdot)) &= \int_\Omega \frac{d_0}{2} f(t,x)^2 + i_{\Omega_{factory}}(x)f(t,x) \, dx + i_{[-\infty , f_{max}]}(f(t,\cdot))
    \end{aligned}
\]
Thus, we can write the cost functional as follows
\begin{equation}\label{eq:var_vacc-quadratic}
    \begin{split}
        {G}((\rho_i,m_i)_{i\in\mathbb{S}},f) &=
         \int_\Omega \sum_{i=S,I,V}\frac{a_i}{2}\rho_i(T,\cdot)^2 \, dx \\
        & + \int^T_0  \int_\Omega \sum_{i = S,I,R} F_i(\rho_i, m_i) \,dx\, dt +\int^T_{T'} \int_\Omega F_V(\rho_V,m_V)\,dx\, dt\\
        & + \int^T_0 \int_\Omega \frac{d_P}{2} (\rho_S+\rho_I+\rho_R)^2 + \frac{d_V}{2} \rho_V^2 \,dx\,dt \\
        & + \int^{T'}_0 \int_\Omega \frac{d_0}{2} f^2 + i_{\Omega_{factory}} f\, dx \, dt\\
        & + \int^T_0 i_{[-\infty, C_{factory}]}(\rho_V(t,\cdot)) + i_{[-\infty, f_{max}]}(f(t,\cdot))\, dt\\
        & + \frac{\lambda}{2}\int^T_0\int_\Omega f^2 + \sum_{i\in\mathbb{S}} \rho_i^2 + |m_i|^2 \,dx\,dt.
    \end{split}
\end{equation}
where $a_i$, $d_P$, $d_V$, $d_0$ are nonnegative constants. With this cost functional, we find explicit formula for each variable $\rho_i^{(k+1)},m_i^{(k+1)},\phi_i^{(k+1)}$ $(i\in\mathbb{S}), f^{(k+1)}$.

\begin{proposition}\label{prop:explicit-formula}
The variables $\rho_i^{(k+1)},m_i^{(k+1)},\phi_i^{(k+1)}$ ($i\in\mathbb{S}$), and $f^{(k+1)}$ from the Algorithm~\ref{alg:gproxPDHG} satisfy the following explicit formulas:
\begin{align*}
    \rho_S^{(k+1)} &= root_+\Biggl(\frac{\tau}{1+ \tau (d_P + \lambda) } \biggl(\partial_t\phi_S^{(k)} + \frac{\eta_S^2}{2} \Delta\phi_S^{(k)} - \frac{1}{\tau} \rho_S^{(k)} + \beta\left( \phi_I^{(k)} - \phi_S^{(k)} \right) K*\rho_I^{(k)} \\
    & \hspace{2cm} + \rho_V^{(k)} \left( \theta_1(\phi_R^{(k)} - \phi_S^{(k)}) - \theta_2 \phi_V^{(k)} \right) + d_P (\rho_I^{(k)} + \rho_R^{(k)})\biggl), 0, - \frac{\tau \alpha_S |m_S^{(k)}|^2}{2(1+\tau (d_P + \lambda))}
    \Biggl)\\
    \rho_I^{(k+1)} &= root_+\Biggl(\frac{\tau}{1+ \tau (d_P + \lambda) } \biggl(\partial_t\phi_I^{(k)} + \frac{\eta_I^2}{2} \Delta\phi_I^{(k)} - \frac{1}{\tau} \rho_I^{(k)} + \beta K * \left( \rho_S^{(k)}(\phi_I^{(k)} - \phi_S^{(k)}) \right)\\
    & \hspace{4.5cm}  + \gamma (\phi_R^{(k)} - \phi_I^{(k)}) + d_P (\rho_S^{(k)} + \rho_R^{(k)})\biggl), 0, - \frac{\tau \alpha_I |m_I^{(k)}|^2}{2(1+\tau (d_P + \lambda))}
    \Biggl)
    \Biggl)\\
    \rho_R^{(k+1)} &= root_+\Biggl(\frac{\tau}{1+ \tau (d_P + \lambda) } \biggl(\partial_t\phi_R^{(k)} + \frac{\eta_R^2}{2} \Delta\phi_R^{(k)} - \frac{1}{\tau} \rho_R^{(k)} + d_P (\rho_S^{(k)} + \rho_I^{(k)})\biggl), 0, - \frac{\tau \alpha_R |m_R^{(k)}|^2}{2(1+\tau (d_P + \lambda))}
    \Biggl)\\
    \rho_V^{(k+1)} &= \min\left( C_{factory}, \frac{\tau}{1 + \tau (d_V+\lambda)} \Bigl( - \partial_t \phi_V^{(k)} - \rho_S^{(k)} (\theta_1(\phi_R^{(k)}-\phi_S^{(k)})-\theta_2\phi_V^{(k)}) + \frac{1}{\tau} \rho_V^{(k)} \Bigr)\right),\\
    & \hspace{11.2cm} (t,x)\in[0,T']\times\Omega\\
    \rho_V^{(k+1)} &= root_+\Biggl(\frac{\tau}{1+ \tau (d_V+\lambda) } \biggl(\partial_t\phi_V^{(k)}   + \rho_S(\theta_1(\phi_R-\phi_S)-\theta_2\phi_V) - \frac{1}{\tau} \rho_V^{(k)}  \biggl), 0, - \frac{\tau \alpha_V |m_V^{(k)}|^2}{2(1+\tau (d_V+\lambda))}
    \Biggl),\\
    & \hspace{11.2cm} (t,x)\in(T',T]\times\Omega\\
   m_i^{(k+1)} &= \frac{\rho_i^{(k+1)}}{\tau\alpha_i + (1+\tau \lambda) \rho_i^{(k+1)}} \left( m_i^{(k)} - \tau \nabla \phi_i^{(k)} \right),\indent (i\in \mathbb{S})\\
   f^{(k+1)} &= \min\left(f_{max}, \frac{\tau}{1 + \tau (d_0+\lambda)} \left( \frac{1}{\tau} f^{(k)} - \phi_V^{(k)} \right) \right)\mathcal{X}_{\Omega_{factory}}(x) \\
   \phi_S^{(k+\frac 1 2)} &= \phi_S^{(k)} + \sigma (A_S A_S^T)^{-1} \Bigl( -\partial_t \rho^{(k+1)}_S - \nabla \cdot m^{(k+1)}_S - \beta \rho^{(k+1)}_S K * \rho^{(k+1)}_I - \theta_1 \rho_S^{(k+1)} \rho_V^{(k+1)}  + \frac{\eta_S^2}{2} \Delta \rho^{(k+1)}_S \Bigr)\\
    \phi_I^{(k+\frac 1 2)} &= \phi_I^{(k)} + \sigma (A_I A_I^T)^{-1} \Bigl( -\partial_t \rho^{(k+1)}_I - \nabla \cdot m^{(k+1)}_I + \beta  \rho^{(k+1)}_S K * \rho^{(k+1)}_I - \gamma \rho^{(k+1)}_I  + \frac{\eta_I^2}{2} \Delta \rho^{(k+1)}_I \Bigr)&\\
    \phi_R^{(k+\frac 1 2)} &= \phi_R^{(k)} + \sigma (A_R A_R^T)^{-1} \Bigl( -\partial_t \rho^{(k+1)}_R - \nabla \cdot m^{(k+1)}_R +  \gamma \rho^{(k+1)}_I + \theta_1 \rho_S^{(k+1)} \rho_V^{(k+1)} + \frac{\eta_R^2}{2} \Delta \rho^{(k+1)}_R\Bigr)\\
    \phi_V^{(k+\frac 1 2)} &= \phi_V^{(k)} + \sigma (A_V A_V^T)^{-1} \left( - \partial_t \rho_V^{(k+1)} + f^{(k+1)} \mathcal{X}_{[0,T')}(t) - \nabla \cdot m_V^{(k+1)} \mathcal{X}_{[T',T]}(t) - \theta_1 \rho_S^{(k+1)} \rho_V^{(k+1)} \right)\\
\end{align*}
where $root_+(a,b,c)$ is a positive root of a cubic polynomial $x^3 + a x^2 + b x +c = 0$ and we approximate the $A_iA_i^*$ as follows
    \begin{align*}
        A_S A_S^{T} &= -\partial_{tt} + \frac{\eta_S^4}{4} \Delta^2 - (1 + (\beta+\theta_1) \eta_S^2) \Delta + (\beta+\theta_1)^2\\
        A_I A_I^{T} &= -\partial_{tt} + \frac{\eta_I^4}{4} \Delta^2 - (1 + (\gamma+\beta) \eta_I^2) \Delta + (\gamma + \beta)^2\\
        A_R A_R^{T} &= -\partial_{tt} + \frac{\eta_R^4}{4} \Delta^2 - \Delta\\
        A_V A_V^{T} &= -\partial_{tt} - \Delta + \theta_2^2.
    \end{align*}
\end{proposition}
We use FFTW library to compute $(A_i A_i^T)^{-1}$ ($i\in \mathbb{S}$) and convolution terms by Fast Fourier Transform (FFT), which is $O(n\log n)$ operations per iteration with $n$ being the number of points. Thus, the algorithm takes just $O(n\log n)$ operations per iteration.

\section{Experiments}\label{sec:experiments}
In this section, we present several sets of numerical experiments using the Algorithm~\ref{alg:gproxPDHG} with various parameters. We wrote C++ codes to run the numerical experiments. Let $\Omega = [0,1]^2$ be a unit square in $\mathbb{R}^2$ and the terminal time $T=1$. The domain $[0,1]\times\Omega$ is discretized with the regular Cartesian grid below.
\[\Delta x_1 = \frac{1}{N_{x_1}},\quad \Delta x_2 = \frac{1}{N_{x_2}},\quad \Delta t = \frac{1}{N_t-1}\]
\begin{equation*}
    \begin{aligned}
        x_{kl} &= \left( (k+0.5) \Delta x_1, (l+0.5)\Delta x_2 \right), && k = 0,\cdots,N_{x_1}-1,\quad l = 0,\cdots,N_{x_2}-1\\
        t_n &= n \Delta t , && n = 0,\cdots,N_t-1
    \end{aligned}
\end{equation*}
where $N_{x_1}$, $N_{x_2}$ are the number of discretized points in space and $N_t$ is the number of discretized points in time.
For all the experiments, we use the same set of parameters,
\begin{gather*}
    \alpha_S = 10, \quad \alpha_I = 30, \quad \alpha_R = 20, \quad \alpha_V = 0.005\\
    a_S = 2, \quad a_I = 2, \quad a_R = 0.001, \quad a_V = 0.1\\
    T'=0.5,\quad\sigma=0.01,\quad d_P = 0.4,\quad d_V = 0.4, \quad d_0 = 0.01\\
    \theta_2=0.9 \quad \eta_i = 0.01 \quad(i\in \mathbb{S}).
\end{gather*}
By setting a higher value for $\alpha_I$, we penalize the infected population's movement more than other populations. Considering the immobility of the infected individuals, this is a reasonable choice in terms of real-world applications. By setting $T'=1/2$, the solution will produce the vaccines during $0\leq t < 1/2$ and deliver them during $1/2 \leq t \leq 1$. Furthermore, we fix the parameters for the infection rate and recovery rate 
\[
    \beta = 0.8,\quad \gamma=0.1.
\]
The paper~\cite{lee2020controlling} describes how the parameters $\beta$ and $\gamma$ affect the propagation of the populations. In this paper, we focus on the vaccine productions and distributions. Recall that from the formulation~\eqref{eq:var_vacc-quadratic}, we have terminal functionals
\[\mathcal{E}_i(\rho_i(T,\cdot)) = \int_\Omega \frac{a_i}{2}\rho_i(T,x)^2\, dx,\quad i\in\{S,I,V\}. \]
Thus, the solution to the problem has to minimize the total number of susceptible, infected, and vaccines at the terminal time $T$. The solution reduces the total number of infected by recovering them with a rate $\gamma$ and decreases the total number of susceptible by transforming the susceptible to the infected with a rate $\beta$ or to the recovered with a rate $\theta_1$ (Figure~\ref{fig:SIRflow}). If the $\beta$ is large and $\gamma$ is small, the number of infected will grow since there are more inflows from susceptible than the outflows to the recovered. To minimize the total number of the infected, the solution has to vaccinate the susceptible as much as possible to avoid the susceptible becoming infected. Thus, the vaccines need to be produced and delivered to the susceptible efficiently while satisfying the constraint conditions~\eqref{eq:constraint2}.
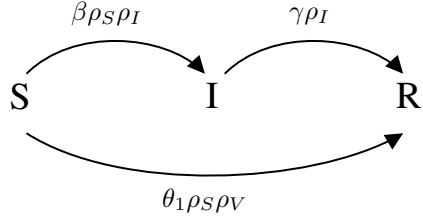
\begin{figure}
    \centering
    \tikzset{every picture/.style={line width=0.75pt}} 
\begin{tikzpicture}[x=0.75pt,y=0.75pt,yscale=-1,xscale=1]

\draw    (80,70) .. controls (103.66,46.32) and (149.19,50.15) .. (168.05,68) ;
\draw [shift={(170,70)}, rotate = 228.18] [fill={rgb, 255:red, 0; green, 0; blue, 0 }  ][line width=0.08]  [draw opacity=0] (8.93,-4.29) -- (0,0) -- (8.93,4.29) -- cycle    ;
\draw    (80,100) .. controls (126.81,130.01) and (222.59,126.58) .. (267.96,101.18) ;
\draw [shift={(270,100)}, rotate = 509.25] [fill={rgb, 255:red, 0; green, 0; blue, 0 }  ][line width=0.08]  [draw opacity=0] (8.93,-4.29) -- (0,0) -- (8.93,4.29) -- cycle    ;
\draw    (180,70) .. controls (203.66,46.32) and (249.19,50.15) .. (268.05,68) ;
\draw [shift={(270,70)}, rotate = 228.18] [fill={rgb, 255:red, 0; green, 0; blue, 0 }  ][line width=0.08]  [draw opacity=0] (8.93,-4.29) -- (0,0) -- (8.93,4.29) -- cycle    ;

\draw (70,72.93) node [anchor=north west][inner sep=0.75pt]  [font=\Large] [align=left] {{\fontfamily{ptm}\selectfont S}};
\draw (169,72.93) node [anchor=north west][inner sep=0.75pt]  [font=\Large] [align=left] {{\fontfamily{ptm}\selectfont I}};
\draw (265,72.93) node [anchor=north west][inner sep=0.75pt]  [font=\Large] [align=left] {{\fontfamily{ptm}\selectfont R}};
\draw (101,32) node [anchor=north west][inner sep=0.75pt]    {$\beta \rho _{S} \rho _{I}$};
\draw (211,35) node [anchor=north west][inner sep=0.75pt]    {$\gamma \rho _{I}$};
\draw (147,126.24) node [anchor=north west][inner sep=0.75pt]    {$\theta _{1} \rho _{S} \rho _{V}$};

\end{tikzpicture}
    \caption{\small Visualization of the flow of three populations. The susceptible transforms to the infected with a rate $\beta$ and the recovered with a rate $\theta_1$. The infected transforms to the recovered with a rate $\gamma$.}
    \label{fig:SIRflow}
\end{figure}

We present two experiments that demonstrate how the various factors in the formulation affect the production and the distribution of vaccines.

\subsection{Experiment~1}\label{subsection:exp0}

In this experiment, we show that Algorithm~\ref{alg:gproxPDHG} converges independent of grid sizes when we use the preconditioning operator defined in Proposition~\ref{prop:explicit-formula}. Consider the initial densities for the $\rho_i$ ($i\in \mathbb{S}$) and the factory location $\Omega_{factory}$ as
\begin{equation}\label{eq:exp1-initial-densities}
    \begin{aligned}
        \rho_S(0,x) &= \left( 2 \exp(-5[(x_1-0.7)^2 + (x_2 - 0.7)^2]) - 1.5 \right)_+\\
        \rho_I(0,x) &= \left( 2 \exp(-5[(x_1-0.7)^2 + (x_2 - 0.7)^2]) - 1.8 \right)_+\\
        \rho_R(0,x) &= 0\\
        \rho_V(0,x) &= 0\\
        \Omega_{factory} &= B_{0.1}(0.3,0.3)
    \end{aligned}
\end{equation}
where $(x)_+ = \max(x,0)$ and $B_r(x)$ is a ball of a radius $r$ centered at $x$. Figure~\ref{fig:my_label} shows the images of initial conditions~\eqref{eq:exp1-initial-densities}.
\begin{figure}[h]
    \centering
    \includegraphics[width=0.5\linewidth]{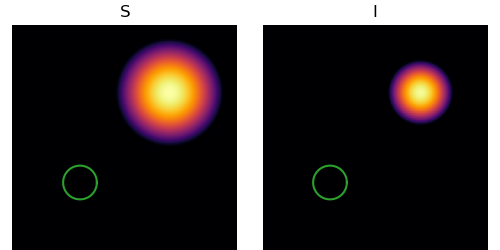}
    \caption{\small Experiment~1: Initial densities of $\rho_S$ (left) and $\rho_I$ (right). The green circle indicates $\Omega_{factory}$.}
    \label{fig:my_label}
\end{figure}
We compute the solution of the SIRV variational problem~\eqref{eq:SIRV-variational-problem} with the above initial conditions using Algorithm~\ref{alg:gproxPDHG}. For simplicity, we assume recovered population density $\rho_R$ does not move. Thus, we use an arbitrary large number for a parameter $\alpha_R = 10^4$ to penalize when $|m_R| > 0$. The rest of the parameters are identical to the parameters defined in the preceding section. We ran four simulations with same initial conditions and same step sizes ($\tau=0.05$, $\sigma=0.2$) with four different grid sizes:\\

\begin{center}
 \begin{tabular}{|| c | c | c ||} 
 \hline
  $N_{x_1}$ & $N_{x_2}$ & $N_t$\\
 \hline
 $32$ & $32$ & $32$\\
 \hline
 $64$ & $64$ & $32$\\
 \hline
 $128$ & $128$ & $32$\\
 \hline
 $256$ & $256$ & $32$\\
 \hline
\end{tabular}
\end{center}
\medskip
The result of the experiment is depicted in Figure~\ref{fig:exp1-convergence-grid-size}. The figure shows the convergence plot of the algorithm with respect to the number of iteration for each grid size. The $x$-axis indicates the iteration number and the $y$-axis indicates the value of the following Lagrangian functional:
\[
    \begin{aligned}
        \tilde{\mathcal{L}}((\rho_i,m_i,\phi_i)_{i\in\mathbb{S}},f) &=
         \int_\Omega \sum_{i=S,I,V}\frac{a_i}{2}\rho_i(T,\cdot)^2 \, dx \\
        & + \int^T_0  \int_\Omega \sum_{i = S,I,R} F_i(\rho_i, m_i) \,dx\, dt +\int^T_{T'} \int_\Omega F_V(\rho_V,m_V)\,dx\, dt\\
        & + \int^T_0 \int_\Omega \frac{d_P}{2} (\rho_S+\rho_I+\rho_R)^2 + \frac{d_V}{2} \rho_V^2 \,dx\,dt + \int^{T'}_0 \int_\Omega \frac{d_0}{2} f^2 \, dx \, dt\\
        & - \int^T_0 \int_\Omega \phi_S \left( \partial_t \rho_S + \nabla \cdot m_S + \beta \rho_S K * \rho_I + \theta_1 \rho_S \rho_V - \frac{\eta_S^2}{2} \Delta \rho_S \right) \, dx \, dt\nonumber\\
        & - \int^T_0 \int_\Omega \phi_I \left( \partial_t \rho_I + \nabla \cdot m_I - \beta \rho_S K * \rho_I + \gamma \rho_I - \frac{\eta_I^2}{2} \Delta \rho_I \right) \, dx \, dt\nonumber\\
        & - \int^T_0 \int_\Omega \phi_R \left( \partial_t \rho_R + \nabla \cdot m_R - \gamma \rho_I - \theta_1 \rho_S \rho_V - \frac{\eta_R^2}{2} \Delta \rho_R \right) \, dx \, dt\nonumber\\
        & - \int^{T}_0 \int_\Omega \phi_V \left( \partial_t \rho_V - f \mathcal{X}_{[0,T')}(t) + \nabla \cdot m_V \mathcal{X}_{[T',T]}(t) + \theta_2 \rho_S \rho_V  \right) \, dx \, dt.
    \end{aligned}
\]
Note that this Lagrangian functional $\tilde{\mathcal{L}}$ is different from~~\eqref{eq:var_vacc} and~\eqref{eq:saddle}. The terms with indicator functions $i_{\Omega_{factory}}$, $i_{[-\infty,t]}$ are removed to avoid representing $+\infty$ numerically. The absence of the terms may explain that the value $\tilde{\mathcal{L}}$ increases in the first $500$ iterations and then decreases afterwards. Figure~\ref{fig:exp1-rhoV-solution} shows the computed solutions at iteration $3000$  from four different spatial grid sizes ($32\times32$, $64\times64$, $128\times128$, $256\times256$). Each row of the figure shows the evolution of a vaccine density $\rho_V$ from time $t=0$ to $t=1$ computed from each grid size.   
These figures clearly show that the algorithm converges to the same saddle point independent of the grid sizes. 
\begin{figure}
    \centering
    \includegraphics[width=0.7\textwidth]{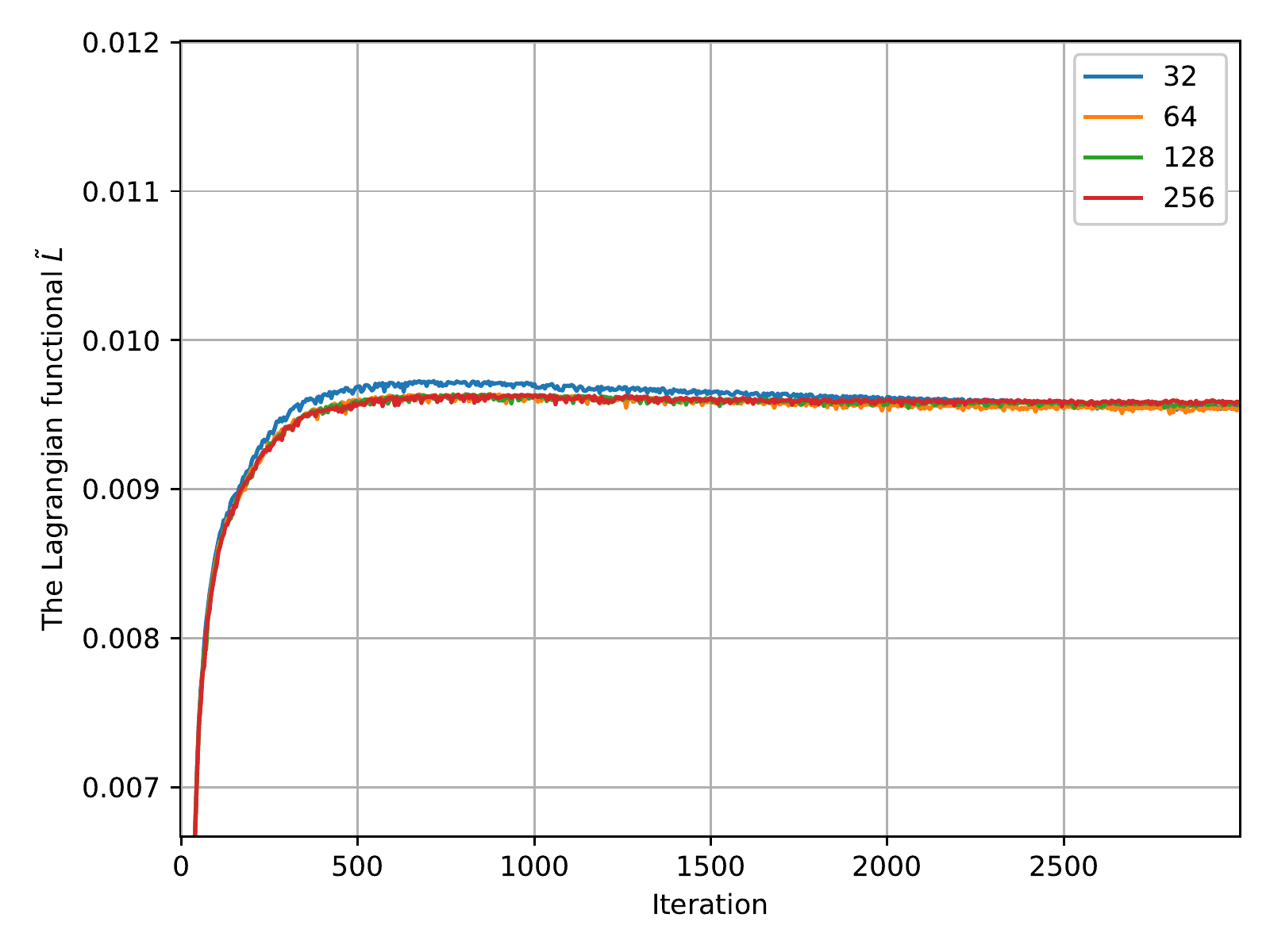}
    \caption{Convergence plot of the algorithm for each grid size ($N_{x_1}=N_{x_2}=32,64,128,256$) with the same step sizes ($\tau=0.05$, $\sigma=0.2$). The plot shows that the convergence of the algorithm is independent of grid sizes.}
    \label{fig:exp1-convergence-grid-size}
\end{figure}

\begin{figure}[h]
    \centering
    \includegraphics[width=0.8\textwidth]{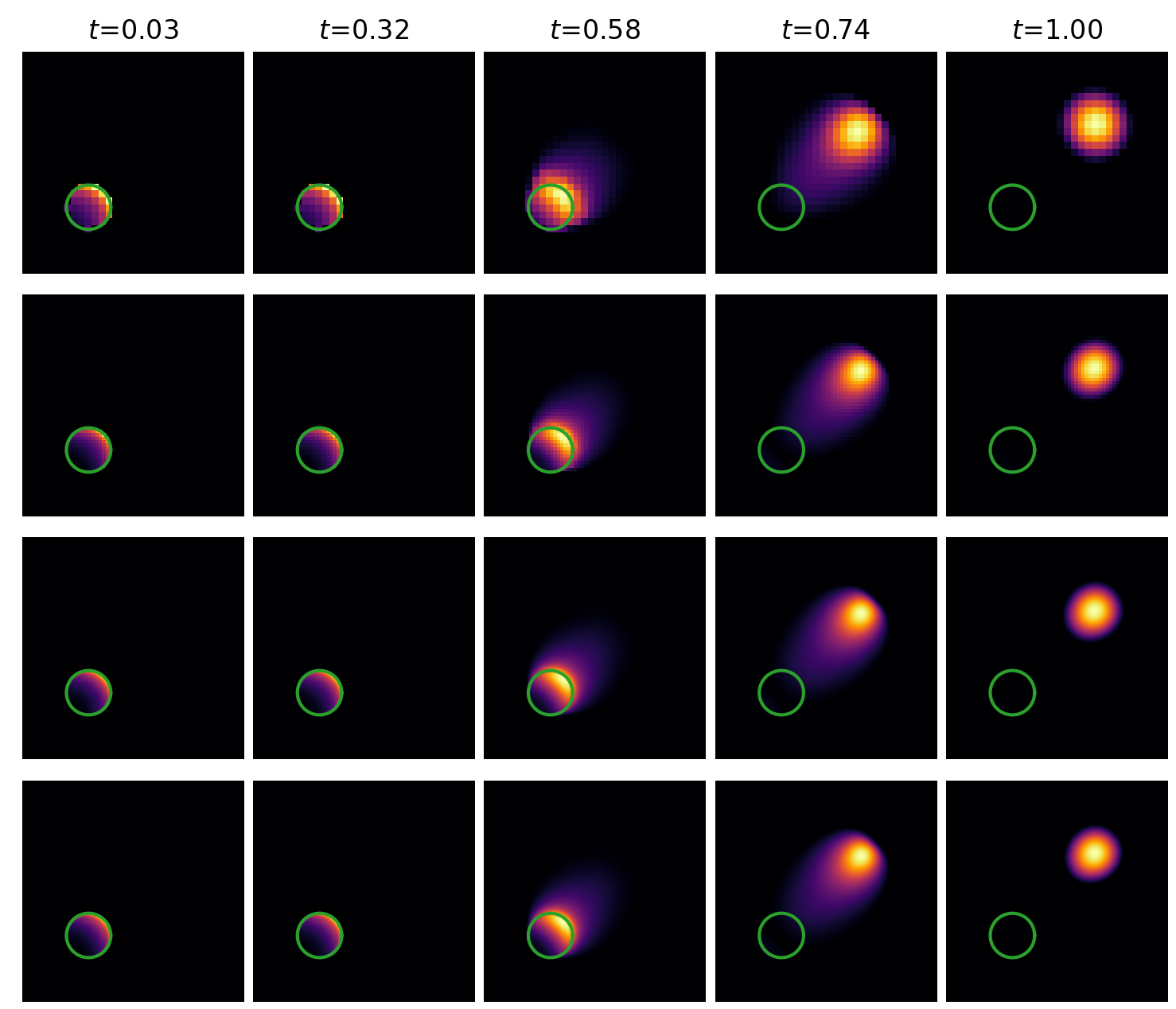}
    \caption{\small Computed solution of a vaccine density variable $\rho_V$ from Experiment~1. Each row shows the evolution of a vaccine density variable from time $t=0$ to $t=1$ with different grid sizes. Row 1: $32\times32$, Row 2: $64\times64$, Row 3: $128\times128$, Row 4: $256\times256$.}
    \label{fig:exp1-rhoV-solution}
\end{figure}

\subsection{Experiment~2}\label{subsection:exp1}

In this experiment, we show how the parameters related to the vaccine density variable $\rho_V$ ($\theta_1, \theta_2, f_{max}, C_{factory}$) affect the solution. We use the same initial densities for $\rho_i$ ($i\in\mathbb{S}$) and $f$ as in Experiment~1. With the initial densities~\eqref{eq:exp1-initial-densities}, we run two simulations with different values for $\theta_1$, $\theta_2$, and $f_{max}$.\\

\begin{center}
 \begin{tabular}{||m{5.5em} | m{3.5em} | m{3.5em} | m{19em}||} 
 \hline
  Parameters & Sim 1 & Sim 2 & Description\\
 \hline\hline
 $\theta_1$ & $0.5$ & $0.9$ & Vaccine efficiency \\ 
 \hline
 $f_{max}$ & $0.5$ & $10$ & Maximum production rate of vaccines   \\
 \hline
  $C_{factory}$ & $0.5$ & $2$ & Maximum amount of vaccines that can be produced at $x\in\Omega$ during $0\leq t \leq \frac12$  \\
 \hline
\end{tabular}
\end{center}

\medskip

\noindent Figure~\ref{fig:exp1-sim1-sim2} shows the comparison between the results from the simulation~1 and the simulation~2. The first three plots (Figure~\ref{fig:exp1-SIR-total-masses}) show the total mass of $\rho_i$ ($i=S,I,R$), i.e.
\[
    \int_\Omega \rho_i(t,x)\, dx,\quad i=S,I,R,\quad t\in[0,1].
\]
and the last plot (Figure~\ref{fig:exp1-vaccine-total-mass}) shows the total mass of $\rho_V$ during $0\leq t \leq \frac12$
\[
    \int_\Omega \rho_V(t,x)\, dx,\quad t\in\left[0,\frac12\right].
\]
The total number of vaccines produced from the simulation~1 is smaller than that from the simulation~2 because the solution cannot produce a large amount of vaccines due to the low production rate $f_{max}$. Furthermore, the solution from the simulation~1 cannot vaccinate a large number of susceptible due to a small $\theta_1$. Thus, there are more susceptible and less recovered at the terminal time in the simulation~1.
\begin{figure}[h]
     \centering
     \begin{subfigure}[b]{1\textwidth}
         \centering
         \includegraphics[width=1\linewidth]{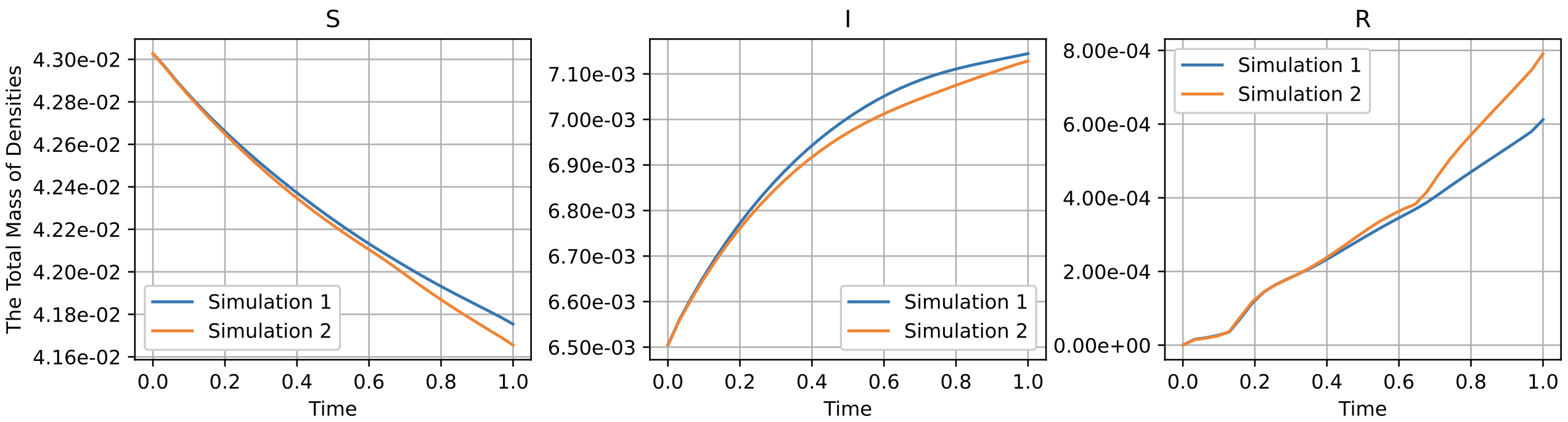}
         \caption{\small The total populations of $\rho_S$, $\rho_I$, $\rho_R$.}
         \label{fig:exp1-SIR-total-masses}
     \end{subfigure}\\
     \begin{subfigure}[b]{1\textwidth}
         \centering
         \includegraphics[width=0.33\textwidth]{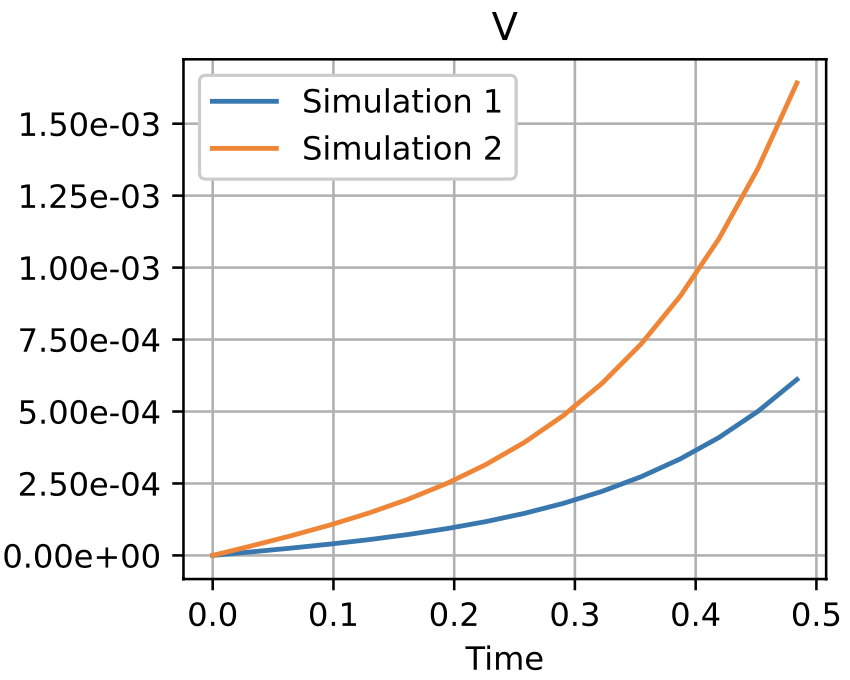}
         \caption{\small The total mass of vaccines produced during $0\leq t \leq 0.5$.}
         \label{fig:exp1-vaccine-total-mass}
     \end{subfigure}
     \caption{\small Experiment~2: The comparison between the results from the simulation~1 and the simulation~2. The first three plots~{(a)} show the total mass of $\rho_i$ ($i=S,I,R$) and the fourth plot~{(b)} shows the total mass of $\rho_V$ produced at the factory area during the production time $0\leq t < 0.5$.}
    \label{fig:exp1-sim1-sim2}
\end{figure}

\subsection{Experiment~3}

This experiment includes the spatial obstacles and shows how the algorithm effectively finds the solution that utilizes the vaccine production and distribution given spatial barriers.  Denote a set $\Omega_{obs} \subset \Omega$ as obstacles. We use the following functionals in the experiment.
\begin{align*}
    \mathcal{G}_P(\rho(t,\cdot)) &= \int_\Omega  \sum_{i\in\{S,I,R\}} \frac{d_i}{2}  \rho_i^2(t,x) + i_{\Omega_{obs}}(x) \left(\sum_{i\in\{S,I,R\}}\rho_i(t,x) \right)\, dx\\
    \mathcal{G}_V(\rho(t,\cdot)) &= \int_\Omega \frac{d_V}{2} \rho_V^2(t,x) + i_{\Omega_{obs}}(x) \rho_V(t,x)\, dx\\
    \mathcal{E}_i(\rho(1,\cdot)) &= \int_\Omega \frac{a_i}{2} \rho_i^2(1,x) + i_{\Omega_{obs}}(x) \rho_i(1,x)\, dx,\quad i\in\{S,I,V\}\\
    \mathcal{E}_R(\rho(1,\cdot)) &= \int_\Omega \frac{a_R}{2} (\rho_R(1,x) - 1)^2 + i_{\Omega_{obs}}(x) \rho_R(1,x)\, dx.
\end{align*}
The densities $\rho_i$ ($i\in\mathbb{S}$) cannot be positive on $\Omega_{obs}$ due to $i_{\Omega_{obs}}$. Thus, the densities transport while avoiding the obstacle $\Omega_{obs}$. We show two sets of experiments based on this setup.

\subsubsection{Single factory}
We set the initial densities and $\Omega_{factory}$ as follows
\begin{equation*}
    \begin{aligned}
        \rho_S(0,x) &= \left( 2 \exp(-15((x_1-0.2)^2+(x_2-0.5)^2)) - 1.6 \right)_+\\
                    & +\left( 2 \exp(-15((x_1-0.8)^2+(x_2-0.5)^2)) - 1.6 \right)_+\\
        \rho_I(0,x) &= \left( 2 \exp(-15((x_1-0.2)^2+(x_2-0.5)^2)) - 1.8 \right)_+\\
        \rho_R(0,x) &= 0\\
        \rho_V(0,x) &= 0\\
        \Omega_{factory} &= B_{0.075}(0.5,0.5)
    \end{aligned}
\end{equation*}
and fix the parameters 
\[
    \theta_1 = 0.9,\quad f_{max} = 10, \quad C_{factory} = 2.
\]
The initial densities are shown in Figure~\ref{fig:exp2-initial-densities}.
\begin{figure}[h]
    \centering
    \includegraphics[width=0.5\linewidth]{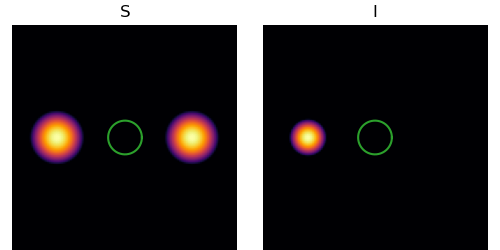}
    \caption{\small Experiment~3: The initial densities $\rho_S$ (left) and $\rho_I$ (right), and the location of the factory (indicated as a green circle).}
    \label{fig:exp2-initial-densities}
\end{figure}

\begin{figure}[h]
    \begin{picture}(450,300)
    \put(0,230){$\rho_S$} 
    \put(0,160){$\rho_I$} 
    \put(0,95){$\rho_R$} 
    \put(0,31){$\rho_V$}
    \put(16,0){\includegraphics[width=0.98\linewidth]{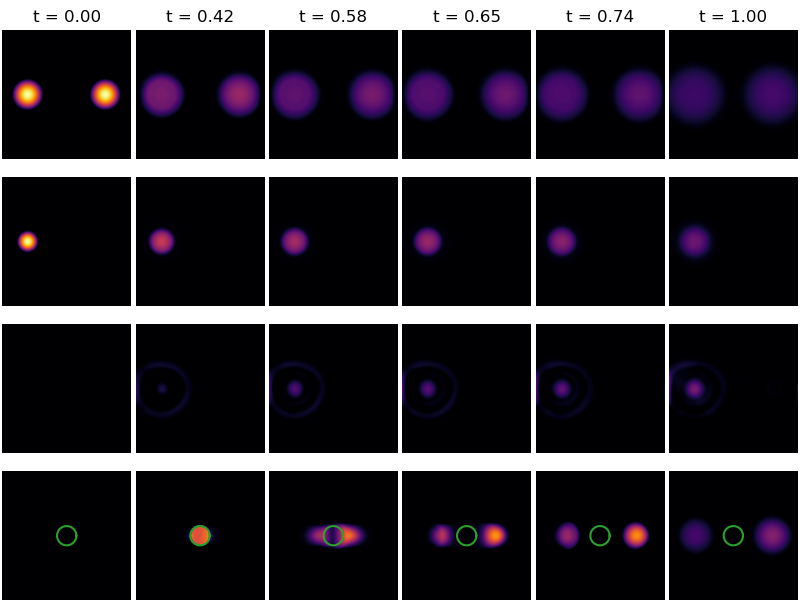}}
    \end{picture}
    \caption{\small Experiment~3: The evolution of densities $\rho_i$ ($i\in\mathbb{S}$) without the obstacle over time $0\leq t \leq 1$. The first row: the susceptible density $\rho_S$. The second row: the infected density $\rho_I$. The third row: the recovered density $\rho_R$. The fourth row: the vaccine density $\rho_V$.}
    \label{fig:exp2-solution-without-obstacle}
\end{figure}

\begin{figure}[h]
    \begin{picture}(450,300)
    \put(0,230){$\rho_S$} 
    \put(0,160){$\rho_I$} 
    \put(0,95){$\rho_R$} 
    \put(0,31){$\rho_V$}
    \put(16,0){\includegraphics[width=0.98\linewidth]{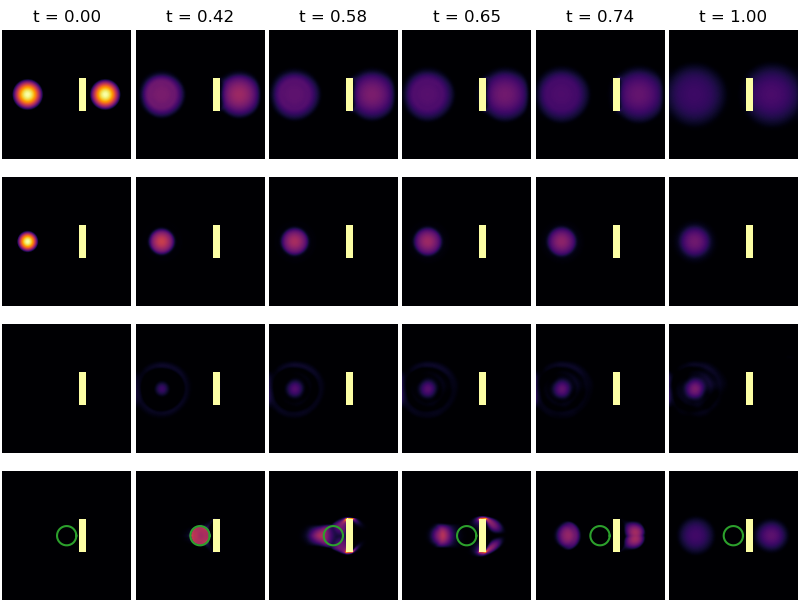}}
    \end{picture}
    \caption{\small Experiment~3: The evolution of densities $\rho_i$ ($i\in\mathbb{S}$) with the obstacle (indicated as a yellow block) over time $0\leq t \leq 1$. The first row: the susceptible density $\rho_S$. The second row: the infected density $\rho_I$. The third row: the recovered density $\rho_R$. The fourth row: the vaccine density $\rho_V$.}
    \label{fig:exp2-solution-with-obstacle}
\end{figure}
Figure~\ref{fig:exp2-solution-without-obstacle} and Figure~\ref{fig:exp2-solution-with-obstacle} show the evolution of densities with and without obstacles, respectively. In both simulations, the density of vaccines $\rho_V$ (the fourth row) transports to the areas where the susceptible people are present. In Figure~\ref{fig:exp2-solution-with-obstacle}, $\rho_V$ transports while avoiding the obstacle at the right. Figure~\ref{fig:exp2-comparison} shows the comparison between these two solutions and how the presence of the obstacle affects the production and delivery of vaccines quantitatively. Figure~\ref{fig:exp2-comparison-left} shows the total mass of the vaccines in the factory area $\Omega_{factory}$ during the production time
\[
    \int_{\Omega_{factory}} \rho_V(t,x)\, dx, \quad t\in[0,0.5).
\]
Figure~\ref{fig:exp2-comparison-right} shows the total mass of the vaccines during the delivery time at the left side and the right side of the domain
\[
\begin{aligned}
    \int_{\Omega \cap \{x_1<0.5\}} \rho_V(t,x)\, dx&, \quad \text{Left}\\
    \int_{\Omega \cap \{x_1\geq 0.5\}} \rho_V(t,x)\, dx&, \quad \text{Right}
\end{aligned}
\]
during $t\in[0.5,1]$. When there is no obstacle, the vaccines are delivered more to the right than to the left (Figure~\ref{fig:exp2-comparison-right}). The number of susceptible people at the left decreases very fast because there are infected people with a high infection rate. When $\rho_V$ starts to transport at time $t=0.5$, the number of susceptible is lower at the left. Thus, the solution distributes fewer vaccines to the left with less susceptible people. When there is an obstacle, $\rho_V$ has to bypass the obstacle to reach the susceptible areas. Thus, the kinetic energy cost during the delivery time $t\in[0.5,1]$ increases at the right. The solution cannot deliver the vaccines as much as the case without the obstacle. It results in a fewer number of vaccines  produced during $t\in[0,0.5)$ (Figure~\ref{fig:exp2-comparison-left}) and delivered to the right during $t\in[0.5,1]$ when there is an obstacle (Figure~\ref{fig:exp2-comparison-right}).

\begin{figure}[h]
    \centering
    \begin{subfigure}[b]{0.5\textwidth}
         \centering
         \includegraphics[width=1\linewidth]{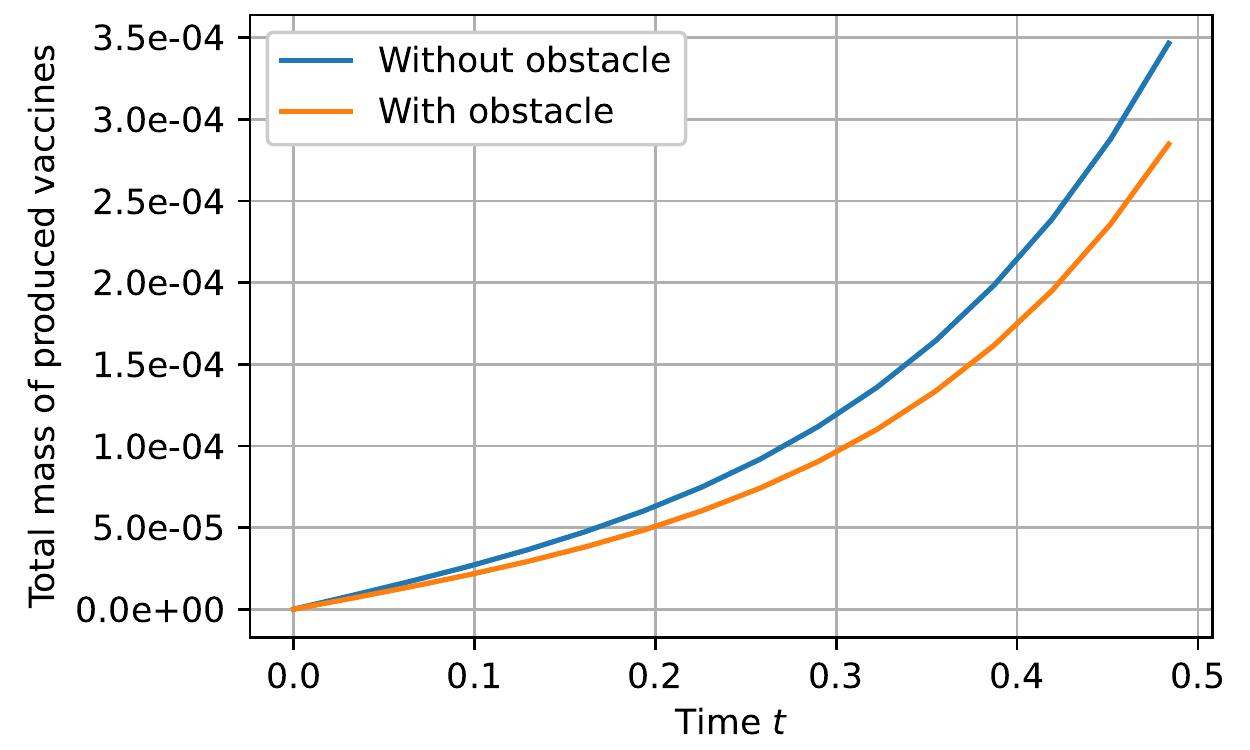}
         \caption{\small The total mass of $\rho_V$ during $t\in[0,0.5)$}
         \label{fig:exp2-comparison-left}
     \end{subfigure}\hfill
     \begin{subfigure}[b]{0.5\textwidth}
         \centering
         \includegraphics[width=1\linewidth]{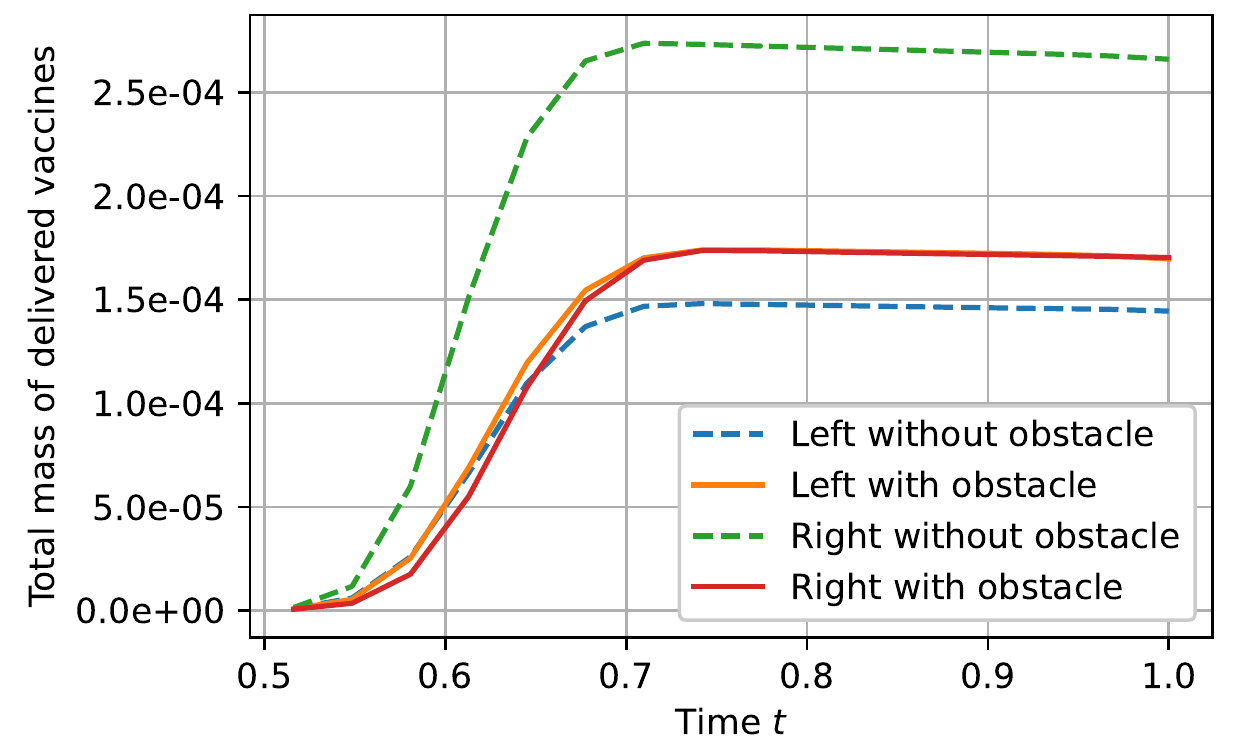}
         \caption{\small The total mass of $\rho_V$ during $t\in[0.5,1]$}
         \label{fig:exp2-comparison-right}
     \end{subfigure}
    \caption{\small Experiment~3: The left plot shows the total mass of vaccine density $\rho_V$ during the production time $t\in[0,0.5)$. The right plot shows the total mass of $\rho_V$ at the left side of the domain $\Omega \cap \{x_1<0.5\}$ and at the right side of the domain $\Omega \cap \{x_1\geq 0.5\}$.}
    \label{fig:exp2-comparison}
\end{figure}

\subsubsection{Multiple factories}

Similar to the previous experiment, we show how the obstacles in the spatial domain affect the production and distribution of the vaccines. We use more complex initial densities, an obstacle set $\Omega_{obs}$, and three factory locations in this experiment. We set the initial densities and $\Omega_{factory}$ as follows
\begin{align*}
\rho_S(0,x) &= \left(  2 \exp(-15((x_1-0.8)^2+(x_2-0.8)^2)) - 1.6 \right)_+\\
& +\left(  2 \exp(-15((x_1-0.2)^2+(x_2-0.7)^2)) - 1.6 \right)_+\\
& +\left(  2 \exp(-15((x_1-0.8)^2+(x_2-0.3)^2)) - 1.6 \right)_+\\
& +\left(  2 \exp(-15((x_1-0.2)^2+(x_2-0.2)^2)) - 1.6 \right)_+\\
\rho_I(0,x) &= \left(  2 \exp(-15((x_1-0.2)^2+(x_2-0.7)^2)) - 1.8 \right)_+\\
&+ \left(  2 \exp(-15((x_1-0.2)^2+(x_2-0.2)^2)) - 1.8 \right)_+\\
\rho_R(0,x) &= 0\\
\rho_V(0,x) &= 0\\
\Omega_{factory} &= B_{0.075}(0.5,0.2) \, \cup \, B_{0.075}(0.5,0.5) \, \cup \, B_{0.075}(0.5,0.8)
\end{align*}
and fix the parameters 
\[
\theta_1 = 0.9,\quad f_{max} = 10, \quad C_{factory} = 2.
\]
The initial densities are shown in Figure~\ref{fig:exp3-initial-densities}.
\begin{figure}[h]
	\centering
	\includegraphics[width=0.5\linewidth]{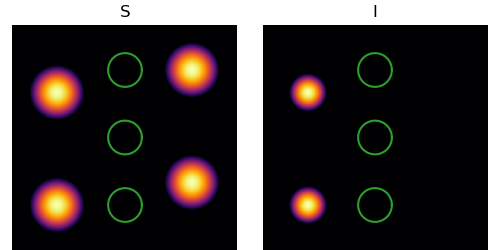}
	\caption{\small Experiment~3: The initial densities $\rho_S$ (left) and $\rho_I$ (right), and the location of the factory (indicated as green circles).}
	\label{fig:exp3-initial-densities}
\end{figure}

\begin{figure}[h]
	\begin{picture}(450,300)
	\put(0,227){$\rho_S$} 
	\put(0,162){$\rho_I$} 
	\put(0,95){$\rho_R$} 
	\put(0,28){$\rho_V$}
	\put(16,0) {\includegraphics[width=0.98\linewidth]{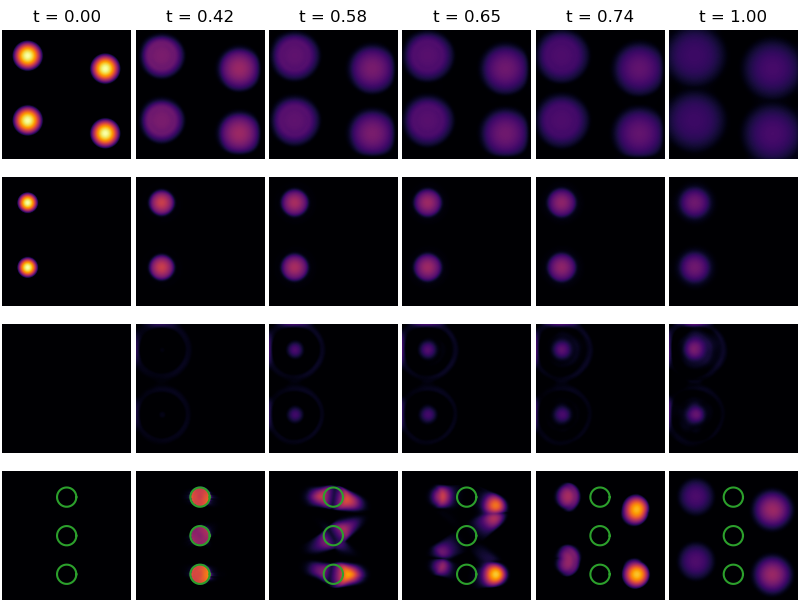}}
	\end{picture}
	\caption{\small Experiment~3: The evolution of densities $\rho_i$ ($i\in\mathbb{S}$) without the obstacle over time $0\leq t \leq 1$. The first row: the susceptible density $\rho_S$. The second row: the infected density $\rho_I$. The third row: the recovered density $\rho_R$. The fourth row: the vaccine density $\rho_V$.}
	\label{fig:exp3-solution-without-obstacle}
\end{figure}

\begin{figure}[h]
	\begin{picture}(450,300)
	\put(0,227){$\rho_S$} 
	\put(0,162){$\rho_I$} 
	\put(0,95){$\rho_R$} 
	\put(0,28){$\rho_V$}
	\put(16,0){\includegraphics[width=0.98\linewidth]{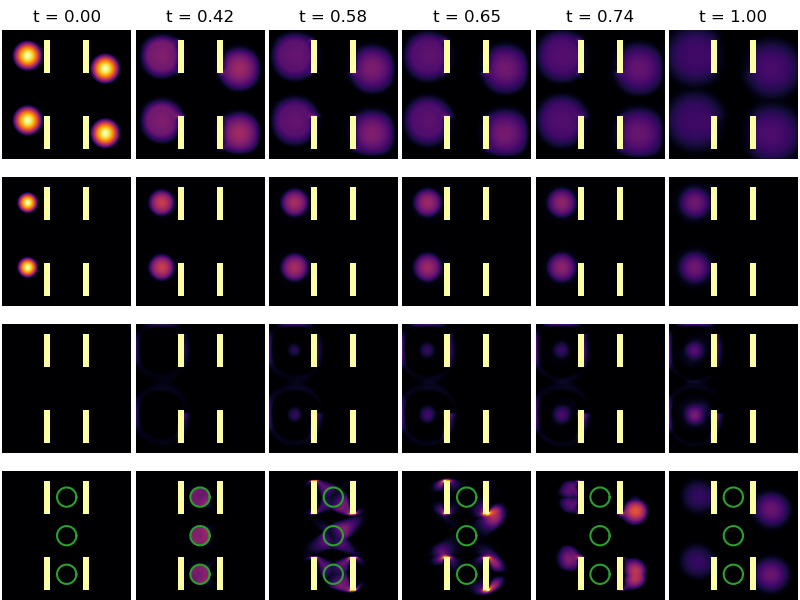}}
	\end{picture}
	\caption{\small Experiment~3: The evolution of densities $\rho_i$ ($i\in\mathbb{S}$) with the obstacle (colored yellow) over time $0\leq t \leq 1$. The first row: the susceptible density $\rho_S$. The second row: the infected density $\rho_I$. The third row: the recovered density $\rho_R$. The fourth row: the vaccine density $\rho_V$.}
	\label{fig:exp3-solution-with-obstacle}
\end{figure}
Figure~\ref{fig:exp3-solution-without-obstacle} and Figure~\ref{fig:exp3-solution-with-obstacle} show the evolution of densities with and without obstacles, respectively. The experiment demonstrates that even with the complex initial densities, the algorithm successfully converges to the reasonable solution that coincides with the previous experiments. The density of vaccines $\rho_V$ (the fourth row) transports to the areas where the susceptible people are present while avoiding the obstacles.\\

Figure~\ref{fig:exp3-comparison-left} shows the total mass of the vaccines produced during the production time at each factory location. Without the obstacles, the total mass of $\rho_V$ at the middle is the lowest at time $0.5$ because the factory at the middle is the farthest away from the susceptible people. It is more efficient to produce the vaccines at the factories closer to the susceptible (the top and the bottom) to reduce the kinetic energy cost during the delivery time $t\in [0.5,1]$. However, the vaccines are produced the most at the middle factory with the obstacles. Since the obstacles block the paths between the top and the bottom factories and the susceptible people, $\rho_V$ has to bypass them to reach the target area. The pathways from the middle factory to the susceptible people are not blocked as much as from the top and the bottom factories. Thus, producing more vaccines at the middle factory is more efficient. \\

Figure~\ref{fig:exp3-comparison-right} shows the total mass of the vaccines during the delivery time at different locations. The lines in the plot represent the following quantities:
\[
\begin{aligned}
\int_{\Omega \cap \{x_1<0.5\} \cap \{x_2\geq0.5\}} \rho_V(t,x)\, dx&, \quad \text{Top Left} &
\int_{\Omega \cap \{x_1\geq 0.5\} \cap \{x_2\geq0.5\}} \rho_V(t,x)\, dx&, \quad \text{Top Right}\\
\int_{\Omega \cap \{x_1<0.5\} \cap \{x_2<0.5\}} \rho_V(t,x)\, dx&, \quad \text{Bottom Left} &
\int_{\Omega \cap \{x_1\geq 0.5\} \cap \{x_2<0.5\}} \rho_V(t,x)\, dx&, \quad \text{Bottom Right}
\end{aligned}
\]
over $t\in[0.5,1]$. With the obstacles, the kinetic energy cost increases since $\rho_V$ has to bypass to reach to the targets when it transports from the top and the bottom factories. As a result, the vaccines are not produced as much as the simulation without the obstacles, and there are less vaccines reached to the targets.

\begin{figure}[h]
	\centering
	\begin{subfigure}[b]{0.55\textwidth}
		\centering
		\includegraphics[width=1\linewidth]{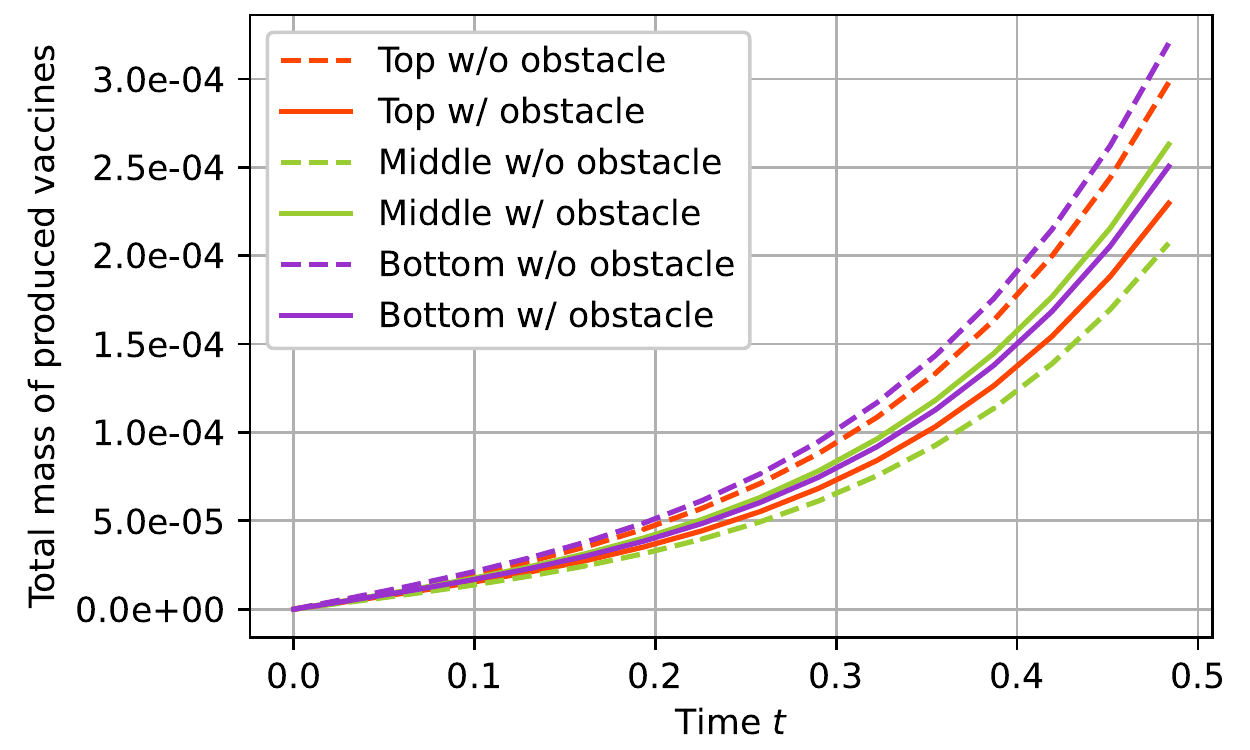}
		\caption{\small The total mass of $\rho_V$ during $t\in[0,0.5)$}
		\label{fig:exp3-comparison-left}
	\end{subfigure}\\
	\begin{subfigure}[b]{0.65\textwidth}
		\centering
		\includegraphics[width=1\linewidth]{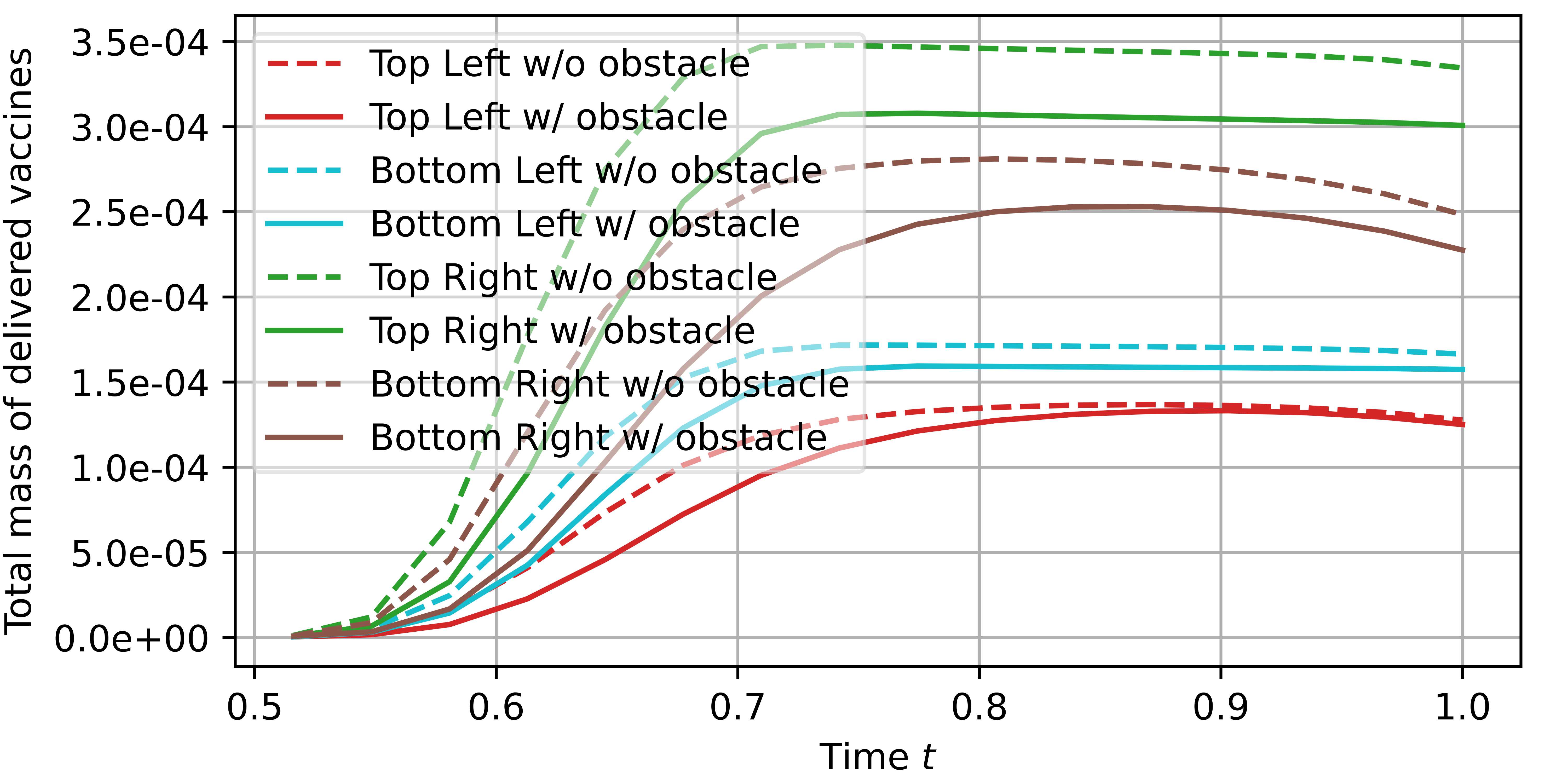}
		\caption{\small The total mass of $\rho_V$ during $t\in[0.5,1]$}
		\label{fig:exp3-comparison-right}
	\end{subfigure}
	\caption{\small Experiment~3: The top plot shows the total mass of vaccine density $\rho_V$ at three factory locations during the production time $t\in[0,0.5)$. The bottom plot shows the total mass of $\rho_V$ at the top left area of the domain $\Omega \cap \{x_1<0.5\} \cap \{x_2\geq0.5\}$, at the bottom left area $\Omega \cap \{x_1<0.5\} \cap \{x_2<0.5\}$, at the top right area $\Omega \cap \{x_1\geq0.5\} \cap \{x_2\geq0.5\}$, and at the bottom right area $\Omega \cap \{x_1\geq 0.5\}\cap \{x_2 < 0.5\}$ during the distribution time $t\in[0.5,1]$.}
	\label{fig:exp3-comparison}
\end{figure}

\subsection{Experiment~4}

This experiment compares the vaccine production strategy generated by the algorithm and the strategy with the fixed rates of production without using the algorithm. The initial densities and $\Omega_{factory}$ are set as follows
\begin{align*}
        \rho_S(0,x) &= \left(  4 \exp(-15((x_1-0.5)^2+(x_2-0.55)^2)) - 1.6 \right)_+\\
        \rho_I(0,x) &= \left(  4 \exp(-15((x_1-0.5)^2+(x_2-0.55)^2)) - 1.8 \right)_+\\
        \rho_R(0,x) &= 0\\
        \rho_V(0,x) &= 0\\
        \Omega_{factory} &= B_{0.04}(0.1,0.3) \, \cup \, B_{0.04}(0.5,0.3) \, \cup \, B_{0.04}(0.9,0.4).
\end{align*}
We fix the parameters 
\[
    \theta_1 = 0.9,\quad f_{max} = 5, \quad C_{factory} = 1.
\]
The initial densities and locations of factories are shown in Figure~\ref{fig:exp3-new-initial-densities}.
\begin{figure}[h]
    \centering
    \includegraphics[width=0.6\linewidth]{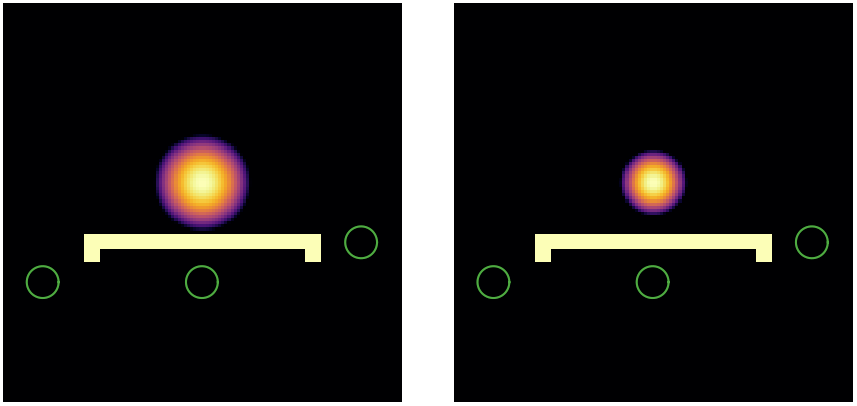}
    \caption{\small Experiment~4: The initial densities $\rho_S$ (left) and $\rho_I$ (right), the location of the factory (indicated as green circles), and the obstacle (colored yellow).}
    \label{fig:exp3-new-initial-densities}
\end{figure}

To fairly compare the effect of the optimal vaccine production strategy, we remove the momentum of $S$, $I$, $R$ groups; thus, removing the spatial movements defined by $m_S$, $m_I$, $m_R$. We consider the following PDEs:
 \begin{equation*}
     \begin{aligned}
         &\partial_t \rho_S  = - \beta \rho_S K * \rho_I + \frac{\eta_S^2}{2} \Delta \rho_S - \theta_1 \rho_V \rho_S && (t,x)\in(0,T)\times\Omega\\
        &\partial_t \rho_I = \beta \rho_S K * \rho_I - \gamma \rho_I + \frac{\eta^2_I}{2} \Delta \rho_I  && (t,x) \in(0,T)\times\Omega\\
        &\partial_t \rho_R  = \gamma \rho_I + \frac{\eta^2_R}{2} \Delta \rho_R + \theta_1 \rho_V \rho_S  && (t,x) \in(0,T)\times\Omega\\
        &\partial_t \rho_V = f(t,x) - \theta_2 \rho_V \rho_S && (t,x) \in(0,T')\times\Omega\\
        &\partial_t \rho_V + \nabla \cdot m_V  = - \theta_2 \rho_V \rho_S  && (t,x) \in[T',T)\times\Omega.
     \end{aligned}
 \end{equation*}
Furthermore, by taking out the momentum terms from $S$, $I$, $R$ groups, the cost functional for this experiment is
\begin{equation}\label{eq:var_vacc-quadratic_no_movements_SIR}
    \begin{split}
        {G}((\rho_i,m_i)_{i\in\mathbb{S}},f) &=
         \int_\Omega \frac{a_V}{2}\rho_V(T,\cdot)^2 \, dx 
         + \int^T_{T'} \int_\Omega F_V(\rho_V,m_V)\,dx\, dt
         + \int^T_0 \int_\Omega \frac{d_V}{2} \rho_V^2 \,dx\,dt \\
        & + \int^{T'}_0 \int_\Omega \frac{d_0}{2} f^2 + i_{\Omega_{factory}} f\, dx \, dt\\
        & + \int^T_0 i_{\{\rho(t,\cdot) \leq C_{factory}\}}(\rho(t,\cdot)) + i_{\{f(t,\cdot) \leq f_{max}\}}(f(t,\cdot))\, dt\\
        & + \frac{\lambda}{2}\int^T_0\int_\Omega f^2 + \rho_V^2 + |m_V|^2 \,dx\,dt.
    \end{split}
\end{equation}
With the PDEs and the cost functionals above, we compare two results. The first result is using the optimal vaccine production and distribution strategy generated by the Algorithm~\ref{alg:gproxPDHG}. The second result is using the fixed vaccine production rate and the algorithm's distribution strategy. In the second result, the factory variable $f$ is fixed as
\[
    f(t,x) = \begin{cases}
     1.2,& (t,x)\in[0,T']\times\Omega_{factory}\\
     0,& (t,x)\in[0,T']\times\Omega\backslash\Omega_{factory}.
    \end{cases} 
\]
Figure~\ref{fig:exp3-new-solution-with-obstacle} shows the comparison between these two results. The result from the fixed production rate is ``without control", and the result from the optimal vaccine production strategy is ``with control". The labels ``left", ``middle", and ``right" are the locations of the factories in Figure~\ref{fig:exp3-new-initial-densities}. 
The solid lines, the result with the same fixed rates of production, show that all three factories produce identical amounts of vaccines. The dotted lines show the least amount of vaccines in the middle factory and much more in the left and right factories. When vaccines produce at the middle factory, one needs to pay more transportation costs because they bypass the obstacles. The obstacle does not block the paths from the left and right factories to the susceptible. Thus, it's an optimal choice to utilize the left and right more than the middle to minimize the transportation costs.

\begin{figure}[h]
\begin{center}
    \includegraphics[width=0.65\linewidth]{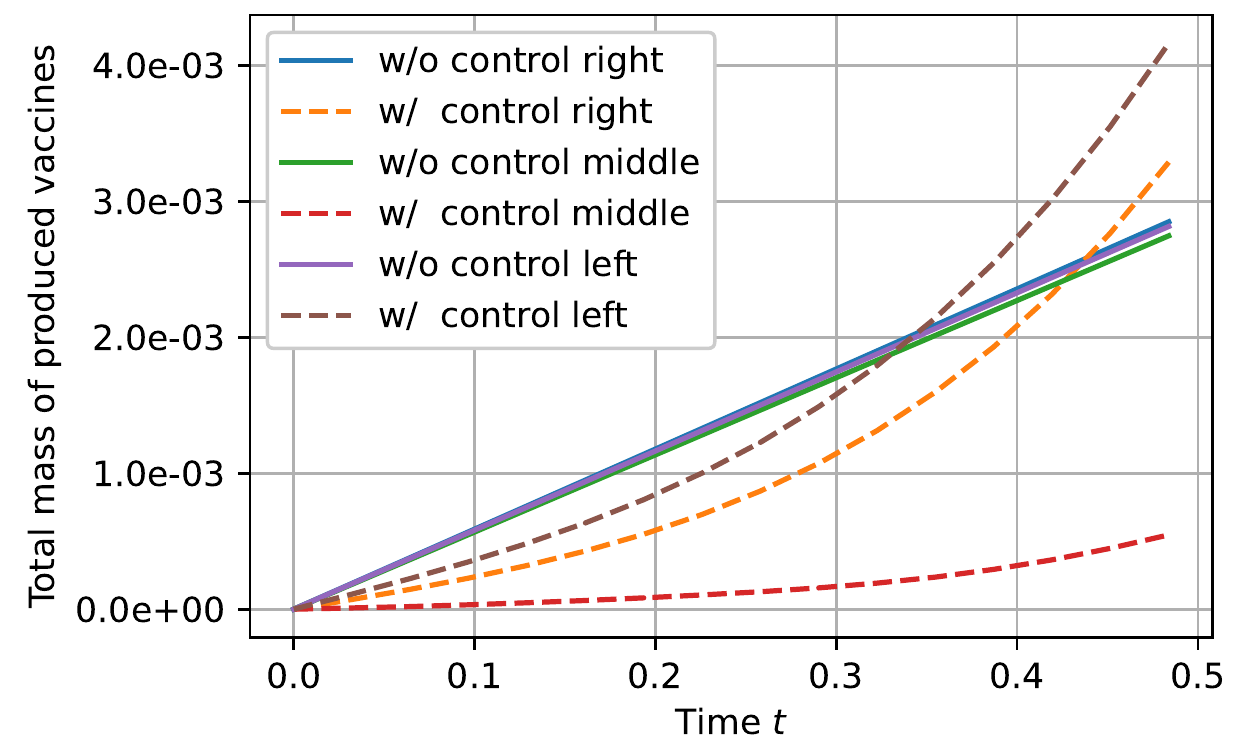}
    \caption{\small Experiment~4: The plot shows the total mass of vaccine densities $\int^t_0\int_\Omega \rho_V\,dx\,dt$ during production $t\in[0,T']$ at each factory location: left, middle, and right. The dotted lines are from the optimal strategy from the Algorithm~\ref{alg:gproxPDHG}, and the solid lines are from the fixed production rates.}
    \label{fig:exp3-new-solution-with-obstacle}
\end{center}
\end{figure}

The table below is the quantitative comparison between the two results.
\def\arraystretch{1.5}
\begin{center}
\begin{tabular}{|| m{6.5em}  m{15em}  m{6em}  m{6em}||} 
 \hline
 Quantity & Description & Algorithm~\ref{alg:gproxPDHG} & Fixed rates \\ [1ex] 
 \hline\hline
 $\int_\Omega \rho_V(\frac12,x)\,dx$ & The total amount of vaccines produced. & $7.997\times10^{-3}$ & $8.411\times10^{-3}$ \\
 \hline
 $\int_\Omega \rho_S(1,x)\,dx$ & The number of susceptible people at the terminal time.  & $1.520\times10^{-2}$ & $1.525\times10^{-2}$ \\ 
 \hline
 $\int_\Omega \rho_I(1,x)\,dx$ & The number of infected people at the terminal time.  & $5.133\times10^{-3}$ & $5.134\times10^{-3}$ \\ 
 \hline
 $\int^1_{\frac12}\int_\Omega \frac{|m_V|^2}{2\rho_V}\,dx\,dt$ & The transportation cost of vaccines.  & $7.339\times10^{-3}$ & $7.544\times10^{-3}$ \\ [1ex] 
 \hline
\end{tabular}
\end{center}

\medskip

\noindent The first row of the table shows that more vaccines are produced with a fixed rate of production. However, the result of the fixed-rate vaccinizes fewer susceptible people; as a result, more infected people at the terminal time. Furthermore, the result from the fixed rate shows higher transportation costs. The algorithm finds the more efficient strategy with fewer vaccines produced.


\section{Appendix}

\begin{proof}[Proof of Proposition~\ref{proposition:KKT-conditions}]
From the saddle point problem~\eqref{eq:saddle}, we can rewrite the problem as
\begingroup
\allowdisplaybreaks
\begin{equation}\label{eq:lagrangian-reformulated}
\begin{split}
    &\inf_{(\rho_i,m_i)_{i\in\mathbb{S}},f} \, \sup_\phi\, {G} ((\rho_i, m_i)_{i\in \mathbb{S}},f)
    -  \int^T_0 \int_\Omega
     \sum_{i\in \{S,I,R\}} \phi_i \left( \partial_t \rho_i + \nabla \cdot m_i - \frac{\eta_i^2}{2} \Delta \rho_i \right) \, dx \, dt\\
        & + \int^T_0 \mathcal{Q}((\rho_i,\phi_i)_{i\in\mathbb{S}})\,dt - \int^T_0\int_\Omega \phi_V \partial_t \rho_V \, dx \, dt
         + \int^{T'}_0 \int_\Omega f \phi_V \, dx \, dt
         - \int^T_{T'} \int_\Omega \phi_V  \nabla \cdot m_V \, dx \, dt 
        \end{split}
\end{equation}
\endgroup
where a function $\mathcal{Q}:(0,T)\times\Omega\rightarrow\mathbb{R}$ is defined as
    \[
        \mathcal{Q}((\rho_i,\phi_i)_{i\in\mathbb{S}}) = \int_\Omega \beta \rho_S (\phi_I - \phi_S) K * \rho_I + \gamma\rho_I (\phi_R - \phi_I) + \rho_S \rho_V\bigl( \theta_1 (\phi_R - \phi_S) - \theta_2 \phi_V)\bigl) \, dx.
    \]
If $((\rho_i, m_i, \phi_i)_{i\in\mathbb{S}},f)$ is the saddle point of the problem, the differential of Lagrangian with respect to $\rho_i$, $m_i$, $\phi_i$ ($i\in \mathbb{S}$), $f$ and $\rho_i(T,\cdot)$ ($i\in\{S,I,V\}$) equal to zero. Thus, from $\frac{\delta \mathcal{L}}{\delta \phi_i}=0$ we have
\begin{equation*}
    \begin{aligned}
        \partial_t\rho_i + \nabla\cdot m_i - \frac{\eta_i^2}{2}\Delta \rho_i + \frac{\delta \mathcal{Q}}{\delta \phi_i}((\rho_i,\phi_i)_{i\in\mathbb{S}}) &= 0 && (t,x)\in(0,T)\times\Omega,\quad i=S,I,R\\
        \partial_t\rho_V - f + \frac{\delta \mathcal{Q}}{\delta \phi_V}((\rho_i,\phi_i)_{i\in\mathbb{S}}) &= 0 && (t,x)\in(0,T')\times\Omega\\
        \partial_t\rho_V + \nabla\cdot m_V + \frac{\delta \mathcal{Q}}{\delta \phi_V}((\rho_i,\phi_i)_{i\in\mathbb{S}}) &= 0 && (t,x)\in(T',T)\times\Omega.
    \end{aligned}
\end{equation*}
Using integration by parts, we reformulate the Lagrangian function~\eqref{eq:lagrangian-reformulated} as follows. 
\begingroup
\allowdisplaybreaks
    \begin{equation*}
        \begin{aligned}
        & \mathcal{L} ((\rho_i,m_i,\phi_i)_{i\in \mathbb{S}},f) \\ =& \sum_{i \in \mathbb{S}}\mathcal{E}_i(\rho_i(T,\cdot)) + \int^T_0 \mathcal{G}_P(\rho_S + \rho_I + \rho_R) + \mathcal{G}_V(\rho_V)\,dt + \int^{T'}_0 \mathcal{G}_0(f(t,\cdot))\, dt\\ 
            & + \sum_{i = S,I,R} \int^T_0 \int_\Omega \frac{\alpha_i |m_i|^2}{2 \rho_i}  + m_i \cdot \nabla \phi_i + \frac{\eta_i^2}{2} \rho_i \Delta \phi_i \, dx \, dt + \sum_{i\in \mathbb{S}} \int^{T}_0 \int_\Omega \rho_i \partial_t \phi_i \, dx \, dt \\
            &+ \int^{T}_{T'} \int_\Omega \frac{\alpha_V |m_V|^2}{2 \rho_V} + m_V \cdot \nabla \phi_V \, dx \, dt  + \int^{T'}_0 \int_\Omega f \phi_V\, dx\,dt
             + \int^T_0 \mathcal{Q}((\rho_i,\phi_i)_{i\in\mathbb{S}})  \, dt\\
            & + \sum_{i = S,I,R, V} \int_\Omega \rho_i(0,x) \phi_i(0,x) - \rho_i(T,x) \phi_i(T,x) dx
        \end{aligned}
    \end{equation*}
\endgroup
    From $\frac{\delta \mathcal{L}}{\delta \rho_i}=0$ ($i\in\{S,I,R\}$),
    \begin{equation*}
        \begin{aligned}
            \frac{\delta \mathcal{G}_P}{\delta \rho_i}(\rho_S+\rho_I+\rho_R) + \frac{\delta \mathcal{Q}}{\delta \rho_i}((\rho_i,\phi_i)_{i\in\mathbb{S}})
            - \frac{\alpha_i |m_i|^2}{2\rho_i^2} + \frac{\eta_i^2}{2}\Delta \phi_i + \partial_t \phi_i &= 0 && (t,x)\in(0,T)\times\Omega
        \end{aligned}
    \end{equation*}
    From $\frac{\delta \mathcal{L}}{\delta \rho_V}=0$,
    \begin{equation*}
        \begin{aligned}
            \frac{\delta \mathcal{G}_V}{\delta \rho_V}(\rho_V) + \frac{\delta \mathcal{Q}}{\delta \rho_V}((\rho_i,\phi_i)_{i\in\mathbb{S}})
             + \partial_t \phi_V &= 0 && (t,x)\in(0,T')\times\Omega\\
            \frac{\delta \mathcal{G}_V}{\delta \rho_V}(\rho_V) + \frac{\delta \mathcal{Q}}{\delta \rho_V}((\rho_i,\phi_i)_{i\in\mathbb{S}}) - \frac{\alpha_V |m_V|^2}{2\rho_V^2} + \partial_t \phi_V &= 0 && (t,x)\in(T',T)\times\Omega.
        \end{aligned}
    \end{equation*}
    From $\frac{\delta \mathcal{L}}{\delta \rho_i(T,\cdot)}=0$ ($i\in \mathbb{S}$),
    \begin{equation*}
        \begin{aligned}
            \frac{\delta \mathcal{E}}{\delta \rho_i(T,\cdot)}(\rho_i(T,\cdot)) = \phi_i(T,\cdot).
        \end{aligned}
    \end{equation*}
    From $\frac{\delta \mathcal{L}}{\delta f}=0$,
    \begin{equation*}
        \begin{aligned}
            \frac{\delta \mathcal{G}_0}{\delta f}(f) + \phi_V = 0, && (t,x) \in (0,T')\times \Omega.
        \end{aligned}
    \end{equation*}
    From $\frac{\delta \mathcal{L}}{\delta m_i}=0$ ($i\in\mathbb{S}$),
    \begin{equation*}
        \begin{aligned}
            \frac{\alpha_i m_i}{\rho_i} &= - \nabla \phi_i && (t,x) \in (0,T)\times\Omega,\quad i\in\{S,I,R\}\\
            \frac{\alpha_V m_V}{\rho_V} &= - \nabla \phi_V && (t,x) \in (0,T')\times\Omega.
        \end{aligned}
    \end{equation*}
    By replacing $\frac{\alpha_i m_i}{\rho_i} = - \nabla \rho_i$ in $\frac{\delta \mathcal{L}}{\delta\rho_i}=0$ and $\frac{\delta \mathcal{L}}{\delta\phi_i}=0$, we derive the result.
\end{proof}


\begin{proof}[Proof of Lemma~\ref{lemma:M-k-bounded}]
    Let $q=(u,p)$. By the definition of $M^{(k)}$, we have
    \[
        \langle q, M^{(k)} q \rangle = \frac{1}{\tau^{(k)}} \|u\|^2_{L^2} + \frac{1}{\sigma^{(k)}} \|p\|_{\mathcal{H}^{(k)}}^2 - 2 \langle u, A_{u^{(k)}}^T p \rangle_{L^2}.
    \]
    Using Young's inequality and Lemma~\ref{lemma:nabla-A-bounded},
    \begin{equation*}
        \begin{split}
            &\leq \left(\frac{1}{\tau^{(k)}} + 1 \right) \|u\|_{L^2}^2 + 
        \left(\frac{1}{\sigma^{(k)}} + 1 \right) \|p\|_{\mathcal{H}^{(k)}}^2\\
            &\leq \left(\frac{1}{\tau^{(k)}} + 1 \right) \|u\|_{L^2}^2 + 
        C^2 \left(\frac{1}{\sigma^{(k)}} + 1 \right) \|p\|_{L^2}^2
            \leq \Theta^2 \|q\|_{L^2}^2.
        \end{split}
    \end{equation*}
    We are left to show the lower bound. Let $\epsilon>0$ be such that $\tau^{(k)} \sigma^{(k)} = (1-\epsilon)^2$. Then using H\"older's inequality,
    \begin{equation*}
        \begin{split}
            \langle q, M^{(k)} q \rangle
            &\geq \frac{1}{\tau^{(k)}} \|u\|^2_{L^2} + \frac{1}{\sigma^{(k)}} \|p\|_{\mathcal{H}^{(k)}}^2 - 2 \|u\|_{L^2} \|p\|_{\mathcal{H}^{(k)}}\\
            &= \frac{1}{\tau^{(k)}} \|u\|^2_{L^2} + \frac{1}{\sigma^{(k)}} \|p\|_{\mathcal{H}^{(k)}}^2 - \frac{2(1-\epsilon)}{\sqrt{\tau^{(k)}\sigma^{(k)}}} \|u\|_{L^2} \|p\|_{\mathcal{H}^{(k)}}.
        \end{split}
    \end{equation*}
    Again, using Young's inequality and Lemma~\ref{lemma:nabla-A-bounded},
    \begin{equation*}
        \begin{split}
            &\geq \frac{\epsilon}{\tau^{(k)}} \|u\|^2_{L^2} + \frac{\epsilon}{\sigma^{(k)}} \|p\|_{\mathcal{H}^{(k)}}^2
            \geq \frac{\epsilon}{\tau^{(k)}} \|u\|^2_{L^2} +  \frac{c^2 \epsilon}{\sigma^{(k)}} \|p\|_{L^2}^2 \geq \theta^2 \|q\|_{L^2}^2.
        \end{split}
    \end{equation*}
    This proves the claim.
\end{proof}

\noindent\textbf{Data Availability Statement:}
All data generated or analysed during this study are included in this published article. 

\noindent\textbf{Conflict of Interest statement}: There is no conflict of interest.

\bibliographystyle{plain}
\bibliography{vac_mfc}
\end{document}